\documentclass[pdflatex,sn-mathphys-num]{sn-jnl}% Math and Physical Sciences Numbered Reference Style
\usepackage[utf8]{inputenc}
\usepackage{xcolor}
\usepackage{amsmath}
\usepackage{amssymb}
\usepackage{graphicx}
\usepackage{longtable}
\usepackage{mathtools}
\mathtoolsset{showonlyrefs}
\usepackage{hyperref}
\hypersetup{colorlinks,linkcolor={blue}, citecolor={blue},}
\usepackage{theoremref}
\usepackage{mleftright}
\usepackage{extpfeil}
\usepackage{comment}

\usepackage{tikz, tikz-cd}
\usetikzlibrary{decorations.pathreplacing}
\usepackage{float}
\usepackage{caption}
\usepackage{ytableau}
\usetikzlibrary{patterns, snakes}
\usepackage{float}
\usepackage{subfig}
\usepackage{pgfplots}
\usepackage{enumitem}

\setlist[enumerate]{font=\small, before=\small}
\setlist[itemize]{font=\small, before=\small}

\theoremstyle{thmstyleone}%
\newtheorem{theorem}{Theorem}%  meant for continuous numbers
\newtheorem{corollary}{Corollary}

\theoremstyle{thmstyletwo}%
\newtheorem{example}{Example}%
\newtheorem{remark}{Remark}%

\theoremstyle{thmstylethree}%
\newtheorem{definition}{Definition}%

\begin{document}

\title[Algebraic Structures on Sets of Partitions]{Algebraic Structures on Sets of Partitions}


\author*[1]{\fnm{Madeline L.} \sur{Dawsey}}\email{mdawsey@uttyler.edu}

\author[1]{\fnm{Megan} \sur{du Preez}}%\email{mdupreez2@patriots.uttyler.edu}

\author[1]{\fnm{Rachel} \sur{Van Surksum}}%\email{rvansurksum@patriots.uttyler.edu}

\affil[1]{\orgdiv{Department of Mathematics}, \orgname{University of Texas at Tyler}, \orgaddress{\street{3900 University Boulevard}, \city{Tyler}, \postcode{75799}, \state{TX}, \country{USA}}}


\abstract{Motivated by Robert Schneider's trailblazing work toward developing a unifying algebraic theory of integer partitions, we explore various binary operations on partitions to identify algebraic structures on sets of partitions. In particular, we discover that several sets of restricted partitions form abelian groups under reduced versions of concatenation, component-wise addition, and component-wise multiplication. One type of restricted partition from a group structure also enjoys a bijection with ordinary partitions of any given size. We extend two partition groups to vector spaces over the finite field $\mathbb{Z}_p$, where $p$ is a prime.  We further discover that partitions are equipped with a commutative ring structure. Finally, we consider subgroups, subspaces, and ideals of our algebraic partition structures to investigate properties of related types of restricted partitions.  The new examples of algebraic structures described in this paper open the door to partition analysis via algebraic tools, decompositions, extensions, and geometry.}




\maketitle

\section*{Acknowledgements}

The authors acknowledge support from NSF Grant DMS-2149921 and the UT Tyler REU. We are grateful for the encouragement and helpful suggestions of Taylor Daniels, Robert Schneider, and Frank Sottile.

\section{Introduction and Statement of Results}%\label{sec1}

A \textit{partition} of an integer $n\geq 0$ is a weakly decreasing sequence of positive integers which sum to $n$. If $\lambda=\left(\lambda_1,\lambda_2,\dots,\lambda_\ell\right)$ is a partition of $n$, then the \textit{size} of $\lambda$ is $|\lambda|:=n$ and the \textit{length} of $\lambda$ is $\ell(\lambda):=\ell$. We also set $\lambda_j:=0$ for all $j>\ell(\lambda)$. We refer to the integers $\lambda_1,\lambda_2,\dots,\lambda_\ell$ in the partition $\lambda=\left(\lambda_1,\lambda_2,\dots,\lambda_\ell\right)$ as the \textit{parts} of $\lambda$, and the \textit{multiplicity} $m_\lambda\left(\lambda_k\right)$ of a part $\lambda_k$ in $\lambda$ is the number of times $\lambda_k$ appears as a part in $\lambda$. Note that the only partition of $n=0$ is $()=\emptyset$, the empty partition, which has no parts, so both its size and its length are zero. We denote the set of all partitions by $\mathcal{P}$. 

To more easily understand the structure of an individual partition, we can use a \textit{Young diagram}, as shown in Figure \ref{partition_example}.  A Young diagram of a partition $\lambda=\left(\lambda_1,\lambda_2,\dots,\lambda_\ell\right)$ is a diagram with $\ell$ left-justified rows of boxes, where the $i$th row from the top contains $\lambda_i$ boxes.

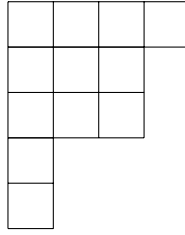
\begin{figure}[h]
    \centering
    \begin{tikzpicture}[scale=0.6,step=1cm]
  \draw (0,0) grid (4,-1);
  \draw (0,-1) grid (3,-3);
  \draw (0,-3) grid (1,-5);
\end{tikzpicture}
    \caption{Young diagram of the partition $\lambda=(4, 3, 3, 1, 1)$}
    \label{partition_example}
\end{figure}

Classically, partitions have been studied either using elementary combinatorial techniques applied to Young diagrams or partition functions, which count the number of partitions of a fixed type and of any given size, or via their generating functions.  Partition generating functions are generally viewed as $q$-series which live somewhere in the universe of complex analytic functions with special types of symmetry, including modular forms, mock modular forms, harmonic Maass forms, and quantum modular forms.  These techniques provide us with some understanding of partitions from a combinatorial, analytic, or even geometric and topological standpoint.  There is a large body of work, always growing, that investigates partitions from these viewpoints (for more details, see \cite{A,HMF,HW}).

In 2018, Robert Schneider successfully completed and defended his Ph.D. thesis titled \emph{Eulerian series, zeta functions, and the arithmetic of partitions} \cite{RS}, in which, among many other innovative ideas, he proposed a new avenue of partition research from an algebraic perspective.  He considered a type of partition multiplication defined by concatenation and then reordering the parts to be weakly decreasing. It is clear that the set $\mathcal{P}$ of partitions under this multiplication operation does not form an algebraic group, since partitions do not have inverses under this operation.  Schneider addresses this issue in his notes\footnote{Schneider, R.: \textit{Notes toward an algebra of partitions}, unpublished lecture notes, Partitions Specialty Seminar, Michigan Technological University. Available at \url{https://pages.mtu.edu/~wjkeith/PartitionsSpecialtySeminar/Algebra_of_multipartitions.4.pdf}.} and instead identifies monoid and semiring structures on $\mathcal{P}$.  To complete the first construction of a group of partitions, he defined the set $\mathcal{P}^-$ of \textit{antipartitions}, which are the inverses of partitions under his product operation, by the relation $\lambda\lambda^-=\emptyset$.  The parts of the antipartition $\lambda^-=\left(\lambda_1^-,\lambda_2^-,\dots,\lambda_\ell^-\right)$ are called \textit{antiparts}, and corresponding parts of $\lambda$ and antiparts of $\lambda^-$ annihilate each other pairwise.  For example, the partition multiplication operation applied to the partition $(7,6,3,3,1)\in\mathcal{P}$ and the antipartition $\left(6^-,4^-,3^-,1^-,1^-\right)\in\mathcal{P}^-$ yields the mixed partition $$(7,6,3,3,1)(6^-,4^-,3^-,1^-,1^-)=(7,6,6^-,4^-,3,3,3^-,1,1^-,1^-)=(7,4^-,3,1^-).$$  If we write the antipartition and antiparts instead as the denominator of a quotient of partitions, then we have the equivalent notation $$(7,4^-,3,1^-)=(7,3)(4,1)^-=(7,3)/(4,1).$$  Using this quotient notation for the antiparts, it is natural to refer to the set $\mathcal{Q}:=\mathcal{P}\cup\mathcal{P}^-$ as the set of \textit{rational partitions}.  Schneider's main theorem in \cite[Appendix B]{RS} is the following.

\begin{theorem}[{\cite[Theorem B.3.2]{RS}}]\label{Schneider1}
The set $\mathcal{Q}$ of rational partitions forms an abelian group under partition multiplication.
\end{theorem}

The arithmetic structure of antipartitions under the partition product operation of Schneider allows us to interpret antipartitions as partitions of negative integers.  Theorem \ref{Schneider1} is the first group structure found on a set related to integer partitions.  Schneider also realized a similar group structure on overpartitions \cite{Overpartitions} and proved that the group of overpartitions is isomorphic to the multiplicative group $\mathbb{Q}$ of rational numbers.  However, the group $\mathcal{Q}$ under Schneider's partition multiplication operation (like the group of overpartitions) includes partitions which do not belong to $\mathcal{P}$.  Enlarging $\mathcal{P}$ to $\mathcal{Q}$ is equivalent to enlarging $\mathbb{N}$ to $\mathbb{Z}$ in each part of each partition, so the group structure comes at the price of no longer being a property of classical integer partitions.

Also in \cite{RS}, Schneider defines a type of partition tensor product $\otimes$ and proves the following theorem, where $\oplus$ now denotes the partition multiplication operation described above.

\begin{theorem}{{\cite[Theorem B.3.5]{RS}}}\label{Schneider_ring_thm}
The set $\mathcal{Q}$ of rational partitions is a commutative ring under the operations of addition $\oplus$ and multiplication $\otimes$.
\end{theorem}

The ring $(\mathcal{Q},\oplus,\otimes)$ is the first ring structure identified on a set related to $\mathcal{P}$.  Note that Schneider's partition multiplication was initially studied as a form of partition product, but in the context of the ring $\mathcal{Q}$, this same operation satisfies the ring axioms for the sum operation.

A few partition operations other than Schneider's have been defined and used to study partitions before. Andrews \cite[Definition 8.11]{A} used the symbol $\oplus$ to define the partition multiplication operation of Schneider twenty years earlier, and he defined this operation in terms of the multiplicities of the parts of the partitions multiplied.  He went on to use this operation to define the theory of partition ideals, which appears to be the very first hint of an algebraic structure on the set of integer partitions in the literature.  In \cite{CDHS}, the first author and her collaborators defined a different operation denoted $\oplus$, specifically on the set of all sequentially congruent partitions (defined by Schneider and Schneider in \cite{SS}).  This operation is defined by left-justified component-wise addition of the parts.  For example, if $\lambda=(5,3,2,2,1)$ and $\gamma=(4,3,1)$, then $\lambda\oplus\gamma=(5+4,3+3,2+1,2+0,1+0)=(9,6,3,2,1)$.  The operation of scalar multiplication was also defined in \cite{CDHS} for any nonnegative integer $c$ and any partition $\lambda=\left(\lambda_1,\lambda_2,\dots,\lambda_\ell\right)$ by $c\lambda=c\left(\lambda_1,\lambda_2,\dots,\lambda_\ell\right)=\left(c\lambda_1,c\lambda_2,\dots,c\lambda_\ell\right)$, so that $c\lambda=\underbrace{\lambda\oplus\lambda\oplus\cdots\oplus\lambda}_{c\text{ times}}$.

Following a presentation on the work in \cite{CDHS} for the online Specialty Seminar in Partition Theory, $q$-Series, and Related Topics hosted by the department of mathematics at Michigan Technological University in 2024, Taylor Daniels asked whether the addition and scalar multiplication operations could lead to a vector space structure on the set of sequentially congruent partitions, and what the implications of such a structure would be for partitions and their operations.  Together, Daniels's question and Schneider's algebraic study of partitions provided the inspiration for the work presented here.

The goal of this work is to identify algebraic structures such as groups, vector spaces, and rings within the set $\mathcal{P}$ of all partitions, without enlarging $\mathcal{P}$ to non-partition multisets of integers. Although it is easy to verify that most of the above defined operations yield monoid structures on $\mathcal{P}$, these monoids do not contain partition inverses under any of the operations. Thus, we modify these operations to obtain several complete algebraic structures within $\mathcal{P}$. 

We refer to the binary operations appearing in our main theorems as:
\begin{itemize}
\item two-color concatenation, denoted by $\cup^2$;
\item concatenation and reduction modulo $k$, denoted by $\cup_k$;
\item component-wise addition and reduction modulo $k$, denoted by $+_k$;
\item component-wise multiplication and zero-reduction modulo $k$, denoted by $\cdot_k^{(0)}$; and
\item component-wise multiplication and reduction modulo $k$, denoted by $\cdot_k$,
\end{itemize}
for any positive integer $k$.  These operations will be defined throughout Sections \ref{section_2color}, \ref{section_groups}, \ref{section_vector_spaces}, and \ref{section_ring}, along with examples and the proofs of our main theorems.  We begin with Theorem \ref{2_colors}, which is essentially a restatement of Theorem \ref{Schneider1} using different notation.  To state Theorem \ref{2_colors}, we must first define the set $\mathcal{P}\times\mathcal{P}$ as the set of two-color partitions, where each part can be one of two colors (see \cite{A1,A2,A3,A4} for more information and recent results on two-color partitions).  Figure \ref{partitionsOf2} shows an example of a two-color partition, where the two colors are white and gray, and white (resp. gray) parts are denoted with a subscript of $w$ (resp. $g$).

    \begin{figure}[h]
        \centering
        \begin{tikzpicture}[scale=0.6,step=1cm]
        		%shade the gray boxes
		\fill[gray!50!white] (3,0) rectangle (5,-1);
		\fill[gray!50!white] (8,-1) rectangle (9,-2);
		\fill[gray!50!white] (10,0) rectangle (11,-2);
		
		%draw all the boxes
        		\draw (0,0) grid (2,-1);
		\draw (3,0) grid (5,-1);
		\draw (6,0) grid (7,-2);
		\draw (8,0) grid (9,-2);
		\draw (10,0) grid (11,-2);
	\end{tikzpicture}
	\caption{Two-color partitions of 2: $\left(2_w\right)$, $\left(2_g\right)$, $\left(1_w,1_w\right)$, $\left(1_w,1_g\right)$, $\left(1_g,1_g\right)$}
        \label{partitionsOf2}
     \end{figure}
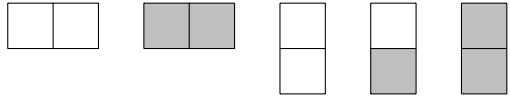
     
We say that a two-color partition $\lambda\in\mathcal{P}\times\mathcal{P}$ is \textit{reduced} if $\lambda$ contains no pair of parts of the same size and different colors, i.e. the white sub-partition of $\lambda$ and the gray sub-partition of $\lambda$ contain no common parts (see Section \ref{section_2color} for further discussion of reduction).  Denote the set of reduced two-color partitions by $(\mathcal{P}\times\mathcal{P})^-$.

\begin{theorem}\label{2_colors}
The set $(\mathcal{P}\times\mathcal{P})^-$ of reduced two-color partitions forms an abelian group under two-color concatenation, $\cup^2$.
\end{theorem}

We will prove in Section \ref{section_2color} that the group structure presented in Theorem \ref{2_colors} is isomorphic to the group structure in Theorem \ref{Schneider1}.  Note that allowing parts to occur in one of two colors, similarly to including both positive and negative integers as parts, significantly enlarges the set $\mathcal{P}$.

More importantly, we prove the following group structures on sets contained in $\mathcal{P}$.  Let $\mathcal{P}_{\cup_k}$ denote the set of partitions where the multiplicity of each part is less than $k$.  From a representation-theoretic perspective, this type of multiplicity-restricted partition is commonly referred to as a \textit{$k$-regular partition}.  Let $\mathcal{P}_k^+$ denote the set of partitions where the difference between any two consecutive parts is less than $k$ and the smallest part is less than $k$.  The partitions in $\mathcal{P}_k^+$ have been called \emph{pattern-avoiding} partitions; in particular, using the terminology and notation of \cite{YZ}, we can think of $\mathcal{P}_k^+$ as containing all partitions with smallest part less than $k$ which avoid the set of patterns $\{[j]:j\geq k\}$, meaning the partitions $\lambda\in\mathcal{P}_k^+$ satisfy the difference condition $\lambda_i-\lambda_{i+1}<k$ for all $1\leq i<\ell(\lambda)$.  In fact, there exists a general algorithm that can enumerate partitions of any size $n$ in this set (see the algorithms developed in \cite{YZ}).

\begin{theorem}\label{group_structures}
Let $k$ be any positive integer.
\begin{enumerate}
\item The set $\mathcal{P}_{\cup_k}$ forms an abelian group under concatenation and reduction modulo $k$, $\cup_k$;
\item The set $\mathcal{P}_k^+$ forms an abelian group under component-wise addition and reduction modulo $k$, $+_k$.
\end{enumerate}
\end{theorem}

For our next group structure, we let $\mathcal{P}_{k,r}^{(0)}$, for any fixed non-negative integer $r$, denote the set of partitions of length $r$ with no parts divisible by $k$, where the difference between any two consecutive parts is less than $k$ and the smallest part is less than $k$.  We note that partitions with no parts divisible by $k$ are classically called \textit{$k$-regular partitions} as well, so $\mathcal{P}_{k,r}^{(0)}$ can be viewed as the set of $k$-regular partitions of length $r$ where the difference between any two consecutive parts is less than $k$ and the smallest part is less than $k$.

\begin{theorem}\label{mult_group_structure}
Let $p$ be a prime, and let $r$ be a non-negative integer.  The set $\mathcal{P}_{p,r}^{(0)}$ forms an abelian group under component-wise multiplication and zero-reduction modulo $p$, $\cdot_p^{(0)}$.
\end{theorem}

Interestingly, our first group structure on the set $\mathcal{P}_{\cup_p}$ of partitions into parts with multiplicity less than $p$, where $p$ is a prime, is closely related to the set of partitions containing no part divisible by $p$, from a combinatorial standpoint.  Glaisher's Theorem \cite{Glaisher} famously implies that the number of partitions in $\mathcal{P}_{\cup_p}$ of any given size $n$ is equal to the number of partitions of $n$ (of any length) with no part divisible by $p$; in fact, Glaisher's Theorem holds for any positive integer $k$.  Note that the explicit bijection that Glaisher constructed between these two sets does not preserve length, so the length restriction on $\mathcal{P}_{p,r}^{(0)}$ is where this relation breaks down.

Two of our partition operations lead to vector space structures on partitions as well.

\begin{theorem}\label{vector_space_structures}
Let $p$ be any prime, and let $\mathbb{Z}_p$ denote the finite field with $p$ elements.  Let $k$ be any positive integer.
\begin{enumerate}
\item The abelian group $\left(\mathcal{P}_{\cup_p},\cup_p\right)$ forms a vector space over $\mathbb{Z}_p$;
\item The abelian group $\left(\mathcal{P}_k^+,+_k\right)$ forms a vector space over $\mathbb{Z}_p$.
\end{enumerate}
\end{theorem}

The most sophisticated type of structure we obtain on partitions is a ring.  The following theorem describes the ring structure.

\begin{theorem}\label{ring_structure}
Let $k$ be any positive integer.  The set $\mathcal{P}_k^+$ forms a commutative ring under component-wise addition and reduction modulo $k$, $+_k$, and component-wise multiplication and reduction modulo $k$, $\cdot_k$.
\end{theorem}

Note that the ring $\left(\mathcal{P}_k^+,+_k,\cdot_k\right)$, for any positive integer $k$, is not a field.  By definition (see Section \ref{section_groups}), these operations require component-wise arithmetic to be reduced modulo $k$, so it is straightforward to see that there are zero divisors if $k$ is not prime.  For example, if $k=6$, then performing component-wise multiplication and reduction modulo 6 on the partitions $(9,4,2)\neq\emptyset$ and $(4,3,3)\neq\emptyset$ yields the product $(36,12,6),$ which reduces modulo 6 to the empty partition $\emptyset$.  However, modulo a prime $p$ and for some fixed non-negative integer $r$, the set $\mathcal{P}_{p,r}^{(0)}$ under the operations of component-wise addition and reduction modulo $p$ and component-wise multiplication and reduction modulo $p$ has no zero divisors, since we impose the additional restriction that any partition containing a part divisible by $p$ is automatically reduced to the empty partition $\emptyset$ (see Section \ref{section_groups} for more details on the property that removes zero divisors modulo $p$).  We hoped that the resulting absence of zero divisors would then lead to a field structure on $\mathcal{P}_{p,r}^{(0)}$.  Unfortunately, with the additional reduction of partitions with parts divisible by $p$ to $\emptyset$, additive inverses are no longer unique, and hence a ring structure is impossible on $\mathcal{P}_{p,r}^{(0)}$ under these two operations.

The rest of this paper is organized primarily by type of algebraic structure. In Section \ref{section_2color}, we define the operation of two-color concatenation, and we prove Theorem \ref{2_colors} as well as an isomorphism between $(\mathcal{P}\times\mathcal{P})^-$ and Schneider's group $\mathcal{Q}$.  In Section \ref{section_groups}, we define concatenation and reduction, component-wise addition and reduction, and component-wise multiplication and zero-reduction; we prove Theorems \ref{group_structures} and \ref{mult_group_structure}; and we investigate partitions contained in various subgroups.  In Section \ref{section_vector_spaces}, we prove Theorem \ref{vector_space_structures}.  In Section \ref{section_ring}, we define component-wise multiplication and reduction, we prove Theorem \ref{ring_structure}, and we investigate partitions contained in various ideals.

%%%%%%%%%%%%%%%%%%%%%%%%%%%%%%%%%%%%%%%

\section{Group Structure on Two-Color Partitions}\label{section_2color}
In this section, we define the operation of two-color concatenation on reduced two-color partitions, and we prove Theorem \ref{2_colors}.

\begin{definition}
We define the unary operation \textit{two-color reduction} on any two-color partition $\lambda\in\mathcal{P}\times\mathcal{P}$ by removing each pair of parts of $\lambda$ of the same size and different colors.  We call a two-color partition $\lambda$ \textit{reduced} if $\lambda$ contains no pair of parts of the same size and different colors.
\end{definition}

Figure \ref{twoColoredConcatenationRule} shows an illustration of reduction for two-color partitions.

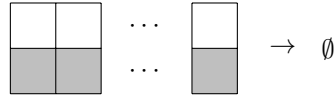
\begin{figure}[h]
    \centering
    \begin{tikzpicture}[scale=0.6,step=1cm]
    \fill[gray!50!white] (0,-1) rectangle (2,-2);
    \fill[gray!50!white] (4,-1) rectangle (5,-2);
    \draw (0,0) grid (2,-2);
    \draw (4,0) grid (5,-2);
    \node at (3,-0.5) {$\cdots$};
    \node at (3,-1.5) {$\cdots$};
    \node at (6,-1) {$\to$};
    \node at (7,-1) {$\emptyset$};
    \end{tikzpicture}
    \caption{Two-color reduction}
    \label{twoColoredConcatenationRule}
\end{figure}

\begin{example}
    The two-color partition $\left(4_w, 3_g, 3_w, 2_g, 2_g, 1_w\right)$ would reduce to the two-color partition $\left(4_w, 2_g, 2_g, 1_w\right)$.  We denote this reduction by $\left(4_w, 3_g, 3_w, 2_g, 2_g, 1_w\right)\to\left(4_w, 2_g, 2_g, 1_w\right),$ and we say the two-color partition $\left(4_w, 2_g, 2_g, 1_w\right)$ is reduced.
\end{example}

Let $(\mathcal{P}\times\mathcal{P})^-$ denote the set of all reduced two-color partitions.  For any reduced two-color partition $\lambda=\left(\lambda_1,\dots,\lambda_\ell\right)\in(\mathcal{P}\times\mathcal{P})^-$, we use the following notation for inverses: if $\lambda_i$ is white, define $\lambda_i^{-1}$ to be the same size but gray; similarly, if $\lambda_i$ is gray, define $\lambda_i^{-1}$ to be the same size but white. Note that due to the reduction process, we get $(\lambda_i,\lambda_i^{-1})\to\emptyset$, as in Figure \ref{twoColoredConcatenationRule}.

\begin{definition}
We define the binary operation \textit{two-color concatenation}, denoted $\cup^2$, on $(\mathcal{P}\times\mathcal{P})^-$ as follows: for any two partitions $\lambda,\pi\in(\mathcal{P}\times\mathcal{P})^-$, define $\lambda\cup^2\pi$ as the concatenation of the parts of $\lambda$ and $\pi$, then reordering to yield a two-color partition in weakly decreasing order by size of its parts, then two-color reduction. 
\end{definition}

Figure \ref{two-color_concatenation} shows an example of two-color concatenation.

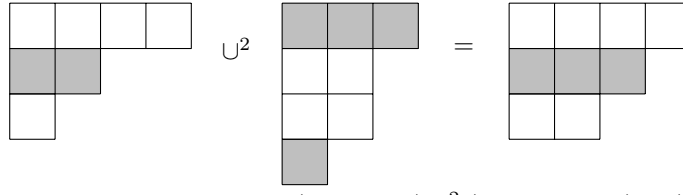
\begin{figure}[h]
\centering
\begin{tikzpicture}[scale=0.6,step=1cm]
\fill[gray!50!white] (0,-1) rectangle (2,-2);
\fill[gray!50!white] (6,0) rectangle (9,-1);
\fill[gray!50!white] (6,-3) rectangle (7,-4);
\fill[gray!50!white] (11,-1) rectangle (14,-2);
\draw (0,0) grid (4,-1);
\draw (0,-1) grid (2,-2);
\draw (0,-2) grid (1,-3);
\node at (5,-1) {$\cup^2$};
\draw (6,0) grid (9,-1);
\draw (6,-1) grid (8,-3);
\draw (6,-3) grid (7,-4);
\node at (10,-1) {$=$};
\draw (11,0) grid (15,-1);
\draw (11,-1) grid (14,-2);
\draw (11,-2) grid (13,-3);
\end{tikzpicture}
\caption{Two-color concatenation: $\left(4_w,2_g,1_w\right)\cup^2\left(3_g,2_w,2_w,1_g\right)=\left(4_w,3_g,2_w\right)$}
\label{two-color_concatenation}
\end{figure}

We now prove Theorem \ref{2_colors} to establish an abelian group structure on $(\mathcal{P}\times\mathcal{P})^-$.

\begin{proof}[Proof of Theorem \ref{2_colors}]
First, we note that $(\mathcal{P}\times\mathcal{P})^-$ is clearly closed under two-color concatenation. Let $\lambda =(\lambda_1,\dots,\lambda_\ell)$, $\pi= (\pi_1,\dots,\pi_p)$, $\tau = (\tau_1,\dots, \tau_t) \in (\mathcal{P}\times\mathcal{P})^-$. 
    %\section{Identity}

    The empty partition $\emptyset$ is the identity element of $\left((\mathcal{P}\times\mathcal{P})^-,\cup^2\right)$:
    \begin{align}
        \lambda \cup^2 \emptyset &=(\lambda_1,\lambda_2,\dots,\lambda_\ell) \cup^2 \emptyset= (\lambda_1,\lambda_2,\dots,\lambda_\ell)= \lambda;
        \\\emptyset \cup^2 \lambda &= \emptyset \cup^2 (\lambda_1,\lambda_2,\dots,\lambda_\ell)=(\lambda_1,\lambda_2,\dots,\lambda_\ell) = \lambda.
    \end{align}
    
    %\section{Inverses}
    Now consider the partition $\lambda^{-1}=\left(\lambda_1^{-1},\lambda_2^{-1},\dots,\lambda_\ell^{-1}\right)$, where each $\lambda_k^{-1}$ is the same size and opposite color of $\lambda_k$ for $1\leq k\leq \ell$. Then we have that $\lambda^{-1}$ is the inverse of $\lambda$ under $\cup^2$:
    \begin{align*}
        \left(\lambda_1,\lambda_2,\dots,\lambda_\ell\right)  \cup^2 \left(\lambda_1^{-1},\lambda_2^{-1},\dots,\lambda_\ell^{-1}\right) = \left(\lambda_1, \lambda_1^{-1}, \lambda_2, \lambda_2^{-1}, \dots, \lambda_\ell, \lambda_\ell^{-1}\right) = \emptyset;\\
       \left(\lambda_1^{-1},\lambda_2^{-1},\dots,\lambda_\ell^{-1}\right)  \cup^2 \left(\lambda_1,\lambda_2,\dots,\lambda_\ell\right) = \left(\lambda_1^{-1}, \lambda_1, \lambda_2^{-1}, \lambda_2, \dots, \lambda_\ell^{-1}, \lambda_\ell\right) = \emptyset.
    \end{align*}
    
    %\section{Associativity}
    Next, we need to show that $\left(\lambda \cup^2 \pi\right) \cup^2 \tau=\lambda \cup^2 \left(\pi \cup^2 \tau\right)$. On the left side, we have 
    \begin{align*}
        \left(\lambda \cup^2 \pi\right) \cup^2 \tau & = \delta \cup^2 \tau = \sigma
        %\\&= (\lambda_1, \dots, \lambda_l, \pi_1,\dots, \pi_p) \cup^2 (\tau_1, \dots, \tau_t)
        %\\&= (\delta_1, \dots, \delta_d) \cup^2 (\tau_1, \dots, \tau_t) = \delta \cup^2 \tau
    \end{align*}
    where $\delta$ is the result of $\lambda\cup^2\pi$ and $\sigma$ is the result of $\delta\cup^2\tau$. On the right side, we have 
    \begin{align*}
        \lambda \cup^2 \left(\pi \cup^2 \tau\right) &= \lambda \cup^2 \zeta = \phi
        %\\&= (\lambda_1, \dots, \lambda_l,) \cup^2 (\pi_1,\dots, \pi_p, \tau_1, \dots, \tau_t)
        %\\&= (\lambda_1, \dots, \lambda_l)  \cup^2 (\zeta_1, \dots, \zeta_z) = \lambda \cup^2 \zeta
    \end{align*}
    where $\zeta$ is the result of $\pi\cup^2\tau$ and $\phi$ is the result of $\lambda\cup^2\zeta$. If $\sigma=\phi=\emptyset$, then we are done. Assume $\sigma$ and $\phi$ are not both empty.  Then $\sigma$ or $\phi$ has at least one part.  Let $\alpha$ be a part in $\sigma$ or $\phi$ with multiplicity $m_\sigma\left(\alpha\right)$ (resp. $m_\phi\left(\alpha\right)$).  Without loss of generality, suppose $m_\sigma\left(\alpha\right)=n\geq1$. Since $\sigma$ is reduced, we have that $\alpha^{-1} \not\in \sigma$; therefore, when we concatenate $\lambda$ with $\pi$ and then with $\tau$, all the parts of size $\alpha^{-1}$ must reduce to $\emptyset$. We consider the following cases: 
    \begin{itemize}
        \item Case 1: $m_\delta\left(\alpha\right)=n+p$ and $m_\tau\left(\alpha^{-1}\right)=p$.
        
        We now consider sub-cases based on how many $\alpha$ parts are in $\lambda$ and $\pi$.
        \begin{itemize}
            \item Case 1.1: $m_\lambda\left(\alpha\right)=n+p+q$ and $m_\pi\left(\alpha^{-1}\right)=q$.
            
            When we concatenate $\pi$ with $\tau$, no pairs of $\alpha,\alpha^{-1}$ reduce. Instead, we have $m_\zeta\left(\alpha^{-1}\right)=m_\pi\left(\alpha^{-1}\right)+m_\tau\left(\alpha^{-1}\right)=q+p$. Then, when we concatenate $\lambda$ with $\zeta$, we see that $p+q$ pairs of $\alpha,\alpha^{-1}$ reduce, leaving $m_\phi\left(\alpha\right)=n$. 
            \item Case 1.2: $m_\lambda\left(\alpha^{-1}\right)=q$ and $m_\pi\left(\alpha\right)=n+p+q$.
            
            When we concatenate $\pi$ with $\tau$, we see that $p$ pairs $\alpha,\alpha^{-1}$ reduce, leaving $m_\zeta\left(\alpha\right)=m_\pi\left(\alpha\right)-m_\tau\left(\alpha^{-1}\right)=n+q$. Then, when we concatenate $\lambda$ with $\zeta$, we have that $q$ pairs $\alpha,\alpha^{-1}$ reduce, leaving $m_\phi\left(\alpha\right)=n$.
            \item Case 1.3: $m_\lambda\left(\alpha\right)=n+p-r$ and $m_\pi\left(\alpha\right)=r$.
            
            When we concatenate $\pi$ with $\tau$, we could get parts of $\alpha$ in $\zeta$ or parts of $\alpha^{-1}$ in $\zeta$, so we break into two more sub-sub-cases: 
            \begin{itemize}
                \item Case 1.3.1: $p\geq r$.
                
                When we concatenate $\pi$ with $\tau$, we see that $r$ pairs $\alpha,\alpha^{-1}$ reduce, leaving $m_\zeta\left(\alpha^{-1}\right)=p-r$. Then, when we concatenate $\lambda$ with $\zeta$, we have that $p-r$ pairs $\alpha,\alpha^{-1}$ reduce, leaving $m_\phi\left(\alpha\right)=n$.
                \item Case 1.3.2: $p< r$.
                
                When we concatenate $\pi$ with $\tau$, we see that $p$ pairs $\alpha,\alpha^{-1}$ reduce, leaving $m_\zeta\left(\alpha\right)=r-p$. Then, when we concatenate $\lambda$ with $\zeta$, no pairs $\alpha,\alpha^{-1}$ reduce, leaving $m_\phi\left(\alpha\right)=n$.
            \end{itemize}
        \end{itemize}
        Thus, all sub-cases of Case 1 end with the conclusion $m_\sigma(\alpha)=m_\phi(\alpha)$.
        \end{itemize}
        
        Counting parts in Cases 2 and 3 is very similar to counting parts in Case 1, so for the remainder of the proof, we simply state the required cases and omit the details.
        \begin{itemize}
        \item Case 2: $m_\delta\left(\alpha^{-1}\right)=p$ and $m_\tau\left(\alpha\right)=n+p$.
        
        We now consider sub-cases based on how many $\alpha$ parts are in $\lambda$ and $\pi$.  All sub-cases below end with the conclusion $m_\sigma(\alpha)=m_\phi(\alpha)$.
        \begin{itemize}
            \item Case 2.1: $m_\lambda\left(\alpha^{-1}\right)=p+q$ and $m_\pi\left(\alpha\right)=q$.
            \item Case 2.2: $m_\lambda\left(\alpha\right)=q$ and $m_\pi\left(\alpha^{-1}\right)=p+q$.
            
            When we concatenate $\pi$ with $\tau$, we could get parts of $\alpha$ in $\zeta$ or parts of $\alpha^{-1}$ in $\zeta$, so we break into two more sub-sub-cases: 
            \begin{itemize}
                \item Case 2.2.1: $q\geq n$.
                \item Case 2.2.2: $q< n$.
            \end{itemize}
            \item Case 2.3: $m_\lambda\left(\alpha^{-1}\right)=p-r$ and $m_\pi\left(\alpha^{-1}\right)=r$. Note that $r\leq p$ here.
        \end{itemize}
        \item Case 3: $m_\delta\left(\alpha\right)=n-m$ and $m_\tau\left(\alpha\right)=m$, with $n\geq m$.
        
        We now break into sub-cases based on how many $\alpha$ parts are in $\lambda$ and $\pi$.  All sub-cases below end with the conclusion $m_\sigma(\alpha)=m_\phi(\alpha)$.
        \begin{itemize}
            \item Case 3.1: $m_\lambda\left(\alpha\right)=n-m+q$ and $m_\pi\left(\alpha^{-1}\right)=q$.
            
            When we concatenate $\pi$ with $\tau$, we could get parts of $\alpha$ in $\zeta$ or parts of $\alpha^{-1}$ in $\zeta$, so we break into two more sub-sub-cases: 
            \begin{itemize}
                \item Case 3.1.1: $q\geq m$.
                \item Case 3.1.2: $q<m$.
            \end{itemize}
            \item Case 3.2: $m_\lambda\left(\alpha^{-1}\right)=q$ and $m_\pi\left(\alpha\right)=n-m+q$.
            \item Case 3.3: $m_\lambda\left(\alpha\right)=n-m-r$ and $m_\pi\left(\alpha\right)=r$.
        \end{itemize}
    \end{itemize}
    We have considered all possible cases that result in $m_\sigma\left(\alpha\right)=n$, for any part $\alpha$ of $\sigma$ or $\phi$.  Therefore, the operation $\cup^2$ is associative in $(\mathcal{P}\times\mathcal{P})^-$.
    
    %\section{Commutativity}
    Finally, we need to show that $\lambda \cup^2 \pi = \pi \cup^2 \lambda$. Consider the concatenation 
    \begin{align*}
        \lambda \cup^2 \pi &= \delta,
        %\\&= (\lambda_1, \dots, \lambda_l, \pi_1, \dots, \pi_p)
        %\\&= (\delta_1, \dots, \delta_d) = \delta
    \end{align*}
    where $\delta$ is again the result of $\lambda\cup^2\pi$, and the concatenation
    \begin{align*}
        \pi \cup^2 \lambda &= \eta,
        %\\&= (\pi_1, \dots, \pi_p, \lambda_1, \dots, \lambda_l, )
        %\\&= (\eta_1, \dots, \eta_e) = \eta
    \end{align*}
    where $\eta$ is the result of $\pi\cup^2\lambda$. 
    If $\delta=\eta=\emptyset$, then we are done. Assume $\delta$ and $\eta$ are not both empty. Then $\delta$ or $\eta$ has at least one part.  Let $\beta$ be a part in $\delta$ of multiplicity $m_\delta(\beta)\geq0$ and in $\eta$ of multiplicity $m_\eta(\beta)\geq0$.  Without loss of generality, suppose $m_\delta(\beta)=n\geq1$. We consider the following cases: 
    \begin{itemize}
        \item Case 1: $m_\lambda(\beta)=n+p$ and $m_\pi\left(\beta^{-1}\right)=p$.
        
        Concatenating $\pi$ with $\lambda$, there are $p$ pairs $\beta,\beta^{-1}$ that reduce, leaving $m_\eta(\beta)=n$.
        \item Case 2: $m_\lambda\left(\beta^{-1}\right)=p$ and $m_\pi(\beta)=n+p$.
        
        Concatenating $\pi$ with $\lambda$, there are $p$ pairs $\beta,\beta^{-1}$ that reduce, leaving $m_\eta(\beta)=n$. 
        \item Case 3: $m_\lambda(\beta)=n-m$ and $m_\pi(\beta)=m$, with $n\geq m$.
        
        Concatenating $\pi$ with $\lambda$, no pairs $\beta,\beta^{-1}$ reduce, leaving $m_\eta(\beta)=n$. 
    \end{itemize}
    We have considered all possible cases that result in $m_\delta(\beta)=n$, for each part $\beta$ appearing in $\delta$ or $\eta$. Therefore, the operation $\cup^2$ on $(\mathcal{P}\times\mathcal{P})^-$ is commutative.  Thus, $\left((\mathcal{P}\times\mathcal{P})^-,\cup^2\right)$ is an abelian group. 
\end{proof}

The abelian group $\left((\mathcal{P}\times\mathcal{P})^-,\cup^2\right)$ is isomorphic to the abelian group $\mathcal{Q}$ of rational partitions under Schneider's partition multiplication operation \cite{RS}.  We construct an explicit isomorphism here.

\begin{theorem}
The group $\left((\mathcal{P}\times\mathcal{P})^-,\cup^2\right)$ is isomorphic to the group $(\mathcal{Q},\cdot)$, where $\cdot$ denotes Schneider's partition multiplication.
\end{theorem}

\begin{proof}
Define the map $\phi:(\mathcal{P}\times\mathcal{P})^-\to\mathcal{Q}$ by $\phi(\lambda)=\lambda'$ for all $\lambda\in(\mathcal{P}\times\mathcal{P})^-$, where $\lambda'$ is the partition obtained from $\lambda$ by writing each white part $\lambda_w$ of $\lambda$ as the positive integer $\lambda_w$ and each gray part $\lambda_g$ of $\lambda$ as the negative integer $-\lambda_g$.  It is clear that for a partition $\lambda\in(\mathcal{P}\times\mathcal{P})^-$, the resulting partition $\lambda'$ is an element of $\mathcal{Q}$, with reduction working the same between parts and antiparts in $\lambda'$ as between white and gray parts in $\lambda$.  It is straightforward to show that $\phi$ is a bijection, and $\phi$ is a homomorphism because $\cup^2$ and $\cdot$ are defined identically and are therefore compatible with each other.
\end{proof}

%%%%%%%%%%%%%%%%%%%%%%%%%%%%%%%%%%%%%%%

\section{Group Structures on Reduced Partitions}\label{section_groups}

In this section, we achieve group structures on the set $\mathcal{P}_{\cup_k}$ of partitions into parts with multiplicity less than $k$, the set $\mathcal{P}_k^+$ of partitions where the difference between any two consecutive parts is less than $k$ and the smallest part is less than $k$, and the set $\mathcal{P}_k^{(0)}$ of partitions with no parts divisible by $k$ where the difference between any two consecutive parts is less than $k$.  The operations required for these group structures each rely on a type of reduction modulo a positive integer $k$.

\subsection{Group Structure on $\mathcal{P}_{\cup_k}$}

We define the reduction required to obtain a group structure on partitions under concatenation.

\begin{definition}
For any positive integer $k$, we define the unary operation \emph{multiplicity reduction modulo $k$} on partitions in $\mathcal{P}$ by removing any $k$ parts of the same size.  We call a partition $\lambda\in\mathcal{P}$ \emph{multiplicity-reduced modulo $k$} if each part of $\lambda$ has multiplicity less than $k$.
\end{definition}

The set $\mathcal{P}_{\cup_k}$ can be viewed as the set of all partitions in $\mathcal{P}$ which are multiplicity-reduced modulo $k$.

Figure \ref{reductionModNExample} shows an illustration of multiplicity reduction modulo $k$. 
    \begin{figure}[h]
        \centering
        \begin{tikzpicture}[scale=0.6,step=1cm]
    \draw (0,0) grid (2,-2);
    \draw (3,0) grid (4,-2);
    \draw (0,-3) grid (2,-4);
    \draw (3,-3) grid (4,-4);
    \node at (2.5,-0.5) {\,$\cdots$};
    \node at (1,-2.3) {$\vdots$};
    \node at (2.5,-2.3) {$\ddots$};
    \node at (5,-2) {$=$};
    \node at (6,-2) {$\emptyset$};
    \node (A) at (0,0) {};
    \node (B) at (0,-4) {};
    \draw[decoration={brace, raise=10pt, amplitude=5pt}, decorate] 
        (B.south west) -- node[left=15pt, align=right] {$k$ rows} (A.north west);
    \end{tikzpicture}
            \caption{Multiplicity reduction modulo $k$}
        \label{reductionModNExample}
    \end{figure}

\begin{definition}
    For any positive integer $k$, we define the binary operation \textit{concatenation and reduction modulo $k$}, denoted $\cup_k$, on $\mathcal{P}$ as follows: for any two partitions $\lambda,\pi\in\mathcal{P}$, define $\lambda\cup_k\pi$ by first concatenating the parts of $\lambda$ and $\pi$, then reordering to yield a partition in weakly decreasing order, then performing multiplicity reduction modulo $k$ on the result. 
    \end{definition}

\begin{example}
For $k=4$, we have $(4, 3, 3, 1)\cup_4(3, 3, 3, 2) = (4, 3, 2, 1)$. Figure \ref{reductionExampleMod4} shows the Young diagrams involved in this example of concatenation and reduction modulo 4. 
    \begin{figure}[h]
        \centering
        \begin{tikzpicture}[scale=0.6,step=1cm]
    \draw (0,0) grid (3,-3);
    \draw (3,0) grid (4,-1);
    \draw (0,-3) grid (1,-4);
    \node at (5,-2) {$\cup_4$};
    \draw (6,0) grid (9,-3);
    \draw (6,-3) grid (8,-4);
    \node at (10,-2) {$=$};
    \draw (11,0) grid (15,-1);
    \draw (11,-1) grid (14,-2);
    \draw (11,-2) grid (13,-3);
    \draw (11,-3) grid (12,-4);
    \end{tikzpicture}
        \caption{Concatenation and reduction modulo $4$: $(4, 3, 3, 1)\cup_4(3, 3, 3, 2) = (4, 3, 2, 1)$}
        \label{reductionExampleMod4}
    \end{figure}
\end{example}

We now prove Theorem \ref{group_structures}, Part 1, to establish an abelian group structure on $\mathcal{P}_{\cup_k}$.

\begin{proof}[Proof of Theorem \ref{group_structures}, Part 1]
    It is clear that $\mathcal{P}_{\cup_k}$ is closed under the operation of concatenation and reduction modulo $k$, $\cup_k$, since reordering and performing multiplicity reduction on the result of concatenation yields a partition with each part occurring less than $k$ times. Let $\lambda = (\lambda_1, \dots, \lambda_\ell)$, $\pi = (\pi_1, \dots, \pi_i)$, and $\tau = (\tau_1, \dots, \tau_t)$ be partitions in $\mathcal{P}_{\cup_k}$. 
    
    The empty partition $\emptyset$ is an element of $\mathcal{P}_{\cup_k}$ and acts as the identity element, because
    \begin{align}
        \lambda \cup_k \emptyset &=(\lambda_1,\lambda_2,\dots,\lambda_\ell) \cup_k \emptyset= (\lambda_1,\lambda_2,\dots,\lambda_\ell)= \lambda;
        \\\emptyset \cup_k \lambda &= \emptyset \cup_k (\lambda_1,\lambda_2,\dots,\lambda_\ell)=(\lambda_1,\lambda_2,\dots,\lambda_\ell) = \lambda.
    \end{align}
    
    Let $\lambda^*$ be the set of distinct parts of $\lambda$, and let $r$ be the cardinality of $\lambda^*$. For each distinct part $\lambda_a\in \lambda^*$ with $m_\lambda(\lambda_a)$ parts in $\lambda$, we note that $1\leq m_\lambda(\lambda_a)<k$ in order for $\lambda$ to be multiplicity reduced modulo $k$.  Consider, for each distinct part $\lambda_a\in\lambda^*$, the partition $\gamma_a := \underbrace{(\lambda_a, \dots, \lambda_a)}_{k-m_\lambda(\lambda_a)\text{ parts}}$. Concatenate every such partition $\gamma_a$ and reorder to get the new partition $\gamma := (\underbrace{\lambda_1, \dots, \lambda_1}_{k-m_\lambda(\lambda_1)\text{ parts}}, \dots, \underbrace{\lambda_r, \dots, \lambda_r}_{k-m_\lambda(\lambda_r)\text{ parts}})$.  Observe that $\gamma\in\mathcal{P}_{\cup_k}$, since $1\leq k-m_\lambda(\lambda_a)<k$ for each distinct part $\lambda_a\in\lambda^*$.  Performing concatenation and reduction modulo $k$ on $\lambda$ and $\gamma$, we have
    \begin{align}
        \lambda \cup_k \gamma &= (\lambda_1,\dots,\lambda_\ell) \cup_k (\underbrace{\lambda_1, \dots, \lambda_1}_{k-m_\lambda(\lambda_1)\text{ parts}}, \dots, \underbrace{\lambda_r, \dots, \lambda_r}_{k-m_\lambda(\lambda_r)\text{ parts}})
        \\&=(\underbrace{\lambda_1, \dots, \lambda_1}_{m_\lambda(\lambda_1)\text{ parts}}, \dots, \underbrace{\lambda_r, \dots, \lambda_r}_{m_\lambda(\lambda_r)\text{ parts}})\cup_k(\underbrace{\lambda_1, \dots, \lambda_1}_{k-m_\lambda(\lambda_1)\text{ parts}}, \dots, \underbrace{\lambda_r, \dots, \lambda_r}_{k-m_\lambda(\lambda_r)\text{ parts}})
        \\&=(\underbrace{\lambda_1, \dots, \lambda_1}_{k \text{ parts}}, \dots, \underbrace{\lambda_r, \dots, \lambda_r}_{k \text{ parts}})
        \\&= \emptyset, 
    \end{align}
    and similarly $\gamma \cup_k \lambda=\emptyset$.  Thus, for each partition $\lambda\in\mathcal{P}_{\cup_k}$, the inverse partition of $\lambda$, denoted by $\gamma$ above, is also an element of $\mathcal{P}_{\cup_k}$.
    
    Next, we want to show that $(\lambda \cup_k \pi) \cup_k \tau = \lambda \cup_k (\pi \cup_k \tau)$. On both sides of the equation, we get the concatenation of all the parts of $\lambda$, $\pi$, and $\tau$, but multiplicity reduction modulo $k$ is performed in a different order on each side.  Thus, it suffices to show that multiplicity reduction modulo $k$ results in reducing the exact same parts on both sides of the equation.  We show this here.
    
    For any part size $\alpha$ in $\lambda$, $\pi$, or $\tau$, we assume that $d_1$ sets of $k$ parts of size $\alpha$ reduce while performing the operations $(\lambda \cup_k \pi) \cup_k \tau$ on the left side, and we assume that $d_2$ sets of $k$ parts of size $\alpha$ reduce while performing the operations $\lambda\cup_k\left(\pi\cup_k\tau\right)$ on the right side. We will show that $d_1=d_2$.  Suppose $m_\lambda(\alpha)=a$, $m_\pi(\alpha)=b$, and $m_\tau(\alpha)=c$. Note that $0\leq a, b, c <k$, and at least one of $a$, $b$, and $c$ is nonzero. When we concatenate $\lambda$ with $\pi$ and reduce, we get $x_1=a+b-f_1k$ parts of size $\alpha$ in $\lambda\cup_k\pi$, where $f_1$ is the number of sets of $k$ parts that reduce. When we then concatenate $\lambda\cup_k\pi$ with $\tau$, we get $y_1=x_1+c-g_1k$ parts of size $\alpha$ in $(\lambda\cup_k\pi)\cup_k\tau$, where $g_1$ is the number of sets of $k$ parts that reduce. Therefore, the number of parts of size $\alpha$ left after both operations is $y_1=(a+b-f_1k)+c-g_1k= a+b+c-(f_1+g_1)k$. Since a total of $d_1$ sets reduce by assumption, we get that $f_1+g_1=d_1$. Thus, $a+b+c\geq d_1k$ and $a+b+c \equiv y_1 \pmod k$ with $0\leq y_1<k$.
    
    On the right side of the equation, we count similarly. When we concatenate $\pi$ with $\tau$, we get $x_2 = b+c- f_2k$ parts of size $\alpha$, where $f_2$ is the number of sets of $k$ parts that reduce. Then, when we concatenate $\lambda$ with $\pi\cup_k\tau$, we get $y_2 = a+x_2 - g_2k$ parts of size $\alpha$, where $g_2$ is the number of sets of $k$ parts that reduce. Therefore, the number of parts of size $\alpha$ left after both operations is $y_2 = a + (b + c -f_2k)-g_2k=a+b+c-(f_2+g_2)k$. Thus, $a+b+c\geq d_2k$ and $a+b+c\equiv y_2\pmod{k}$ with $0\leq y_2<k$.
    
    Combining the two congruences from both sides of the equation, we see that $y_2\equiv y_1\pmod{k}$; since $0\leq y_1,y_2<k$, we must have that $y_1=y_2$. We have now proved that for any part size $\alpha$ in $\lambda$, $\pi$, or $\tau$, the partitions $\left(\lambda\cup_k\pi\right)\cup_k\tau$ and $\lambda\cup_k\left(\pi\cup_k\tau\right)$ contain the same number of parts of size $\alpha$. Since $\alpha$ is an arbitrary part size, we have completed the proof that the operation of concatenation and reduction modulo $k$, $\cup_k$, is associative. 
    
    It is clear that the operation $\cup_k$ is also commutative, because concatenating $\lambda\cup_k\pi$ or $\pi\cup_k\lambda$, in either order, is followed by reordering; the reordering step results in the exact same set of parts, which then results in the exact same partition after multiplicity reduction modulo $k$.
    
    Thus, the set $\mathcal{P}_{\cup_k}$ is an abelian group under concatenation and reduction modulo $k$. 
\end{proof}

Next, we identify a family of subgroups of $\mathcal{P}_{\cup_k}$.

\begin{theorem}
    Fix a positive integer $k$ and a positive divisor $r$ of $k$, and let $\alpha$ be any positive integer.  The set $A_r(\alpha):=\left\{\lambda\in\mathcal{P}_{\cup_k}:r\mid m_\lambda(\alpha)\right\}$ forms a subgroup of $\left(\mathcal{P}_{\cup_k}, \cup_k\right)$. 
\end{theorem}
\begin{proof}
    Let $\lambda,\pi\in A_r(\alpha)$ with $m_\lambda(\alpha)=rs$ and $m_\pi(\alpha)=rt$, for some non-negative integers $s,t$.  When we perform concatenation and reduction modulo $k$ on $\lambda$ and $\pi$, we see that $m_{\lambda\cup_k\pi}(\alpha)=rs+rt-jk = r(s+t)-jk$, where $j$ is a non-negative integer defined so that $0\leq r(s+t)-jk<k$. Since $r\mid k$, we have that $r\mid(r(m+l)-jk)$. Thus, the set $A_r(\alpha)$ is closed under concatenation and reduction modulo $k$. Since the empty partition $\emptyset$ has zero parts of size $\alpha$ and $r\mid0$, we have that $\emptyset\in A_r(\alpha)$. We also have that the inverse partition $\lambda^{-1}$ satisfies $m_{\lambda^{-1}}(\alpha)=k-rs$.  Again, since $r\mid k$, we have that $r\mid(k-rs)$ as well, so the inverse of every partition in $A_r(\alpha)$ is also contained in $A_r(\alpha)$.  This completes the proof that $A_r(\alpha)$ is a subgroup of $\left(\mathcal{P}_{\cup_k},\cup_k\right)$.
\end{proof}

\begin{remark}
We make the following interesting observations regarding the subgroups $A_r(\alpha)$.
\begin{enumerate}
\item For a fixed part $\alpha$ contained in some partition in $\mathcal{P}_{\cup_k}$, the number of subgroups of the form $A_r(\alpha)$ in $\mathcal{P}_{\cup_k}$ is equal to the divisor function $d(k)$, the number of divisors of $k$.
\item For a fixed part $\alpha$ and a fixed divisor $r\mid k$, the subgroup $A_r(\alpha)$ is the set of all partitions into parts with multiplicity less than $k$, where any part of size $\alpha$ occurs a multiple of $r$ times.
\end{enumerate}
\end{remark}

\subsection{Group Structure on $\mathcal{P}_k^+$}

We approach reduction for the operation of component-wise addition differently, by viewing $\mathcal{P}$ as a set of equivalence classes of partitions defined by the congruence classes modulo $k$ of each part. We first show that congruence modulo $k$ in each component, denoted $\sim_k$, is an equivalence relation, and then we prove that the set $\mathcal{P}/_{\sim_k}$ of equivalence classes of partitions is an abelian group.  Finally, we prove that the set $\mathcal{P}_k^+$ of component-wise reduced modulo $k$ partitions is isomorphic to $\mathcal{P}/_{\sim_k}$ to establish an abelian group structure on $\mathcal{P}_k^+$.

\begin{definition}
Let $k$ be a positive integer $k$, and let $\lambda,\pi\in\mathcal{P}$ such that $\lambda=\left(\lambda_1,\lambda_2,\dots,\lambda_\ell\right)$ and $\pi=\left(\pi_1,\pi_2,\dots,\pi_i\right)$ with $\ell\geq i$. We define the relation \textit{component-wise congruence modulo $k$}, denoted $\sim_k$, on the pair $(\lambda,\pi)$ as follows: we say that $\lambda\sim_k\pi$ if and only if $\lambda_j\equiv\pi_j\pmod{k}$ for all $1\leq j\leq i$ and $\lambda_j\equiv0\pmod{k}$ for all $i<j\leq\ell$.  If $\lambda\sim_k\pi$, then we call $\lambda$ and $\pi$ \textit{component-wise congruent modulo $k$}.
\end{definition}

Note that two partitions of different length can be component-wise congruent modulo $k$.  Next, we prove that this relation is, in fact, an equivalence relation on partitions.

\begin{theorem} \label{additionequivalence}
    For any integer $k$, component-wise congruence modulo $k$ is an equivalence relation on $\mathcal{P}$.
\end{theorem}
\begin{proof}
    Let $\lambda =\left(\lambda_1, \lambda_2, \dots, \lambda_\ell\right)$, $\pi = \left(\pi_1, \pi_2, \dots, \pi_i\right)$, $\tau=\left(\tau_1,\tau_2,\dots,\tau_t\right)\in\mathcal{P}$, and suppose without loss of generality that $\ell\geq i\geq t$. The relation $\sim_k$ is reflexive, since $\lambda_j\equiv\lambda_j\pmod{k}$. The relation $\sim_k$ is symmetric, since if $\lambda_j\equiv\pi_j\pmod{k}$ for all $1\leq j\leq i$ and $\lambda_j\equiv0\pmod{k}$ for all $i<j\leq\ell$, then $\pi_j\equiv\lambda_j\pmod{k}$ for all $1\leq j\leq i$ and $\pi_j=0$ for all $i<j\leq\ell$. We have used the convention that $\pi_j=0$ for all $j>\ell(\pi)=i$. Finally, suppose that $\lambda\sim_k\pi$ and $\pi\sim_k\tau$.  Then we have that
    \begin{align*}
    &\lambda_j\equiv\pi_j\pmod{k}\text{ for all }1\leq j\leq i;\\
    &\lambda_j\equiv0\pmod{k}\text{ for all }i<j\leq\ell;\\
    &\pi_j\equiv\tau_j\pmod{k}\text{ for all }1\leq j\leq t;\\
    &\pi_j\equiv0\pmod{k}\text{ for all }t<j\leq i.
    \end{align*}
Considering all four conditions simultaneously, we have that the parts of $\lambda$ and $\tau$ satisfy
\begin{align*}
&\lambda_j\equiv\pi_j\equiv\tau_j\pmod{k}\text{ for all }1\leq j\leq t;\\
&\lambda_j\equiv0\pmod{k}\text{ for all }t<j\leq\ell.
\end{align*}
Then $\lambda\sim_k\tau$, and therefore $\sim_k$ is transitive.  Thus, we have shown that $\sim_k$ is an equivalence relation on $\mathcal{P}$.
\end{proof} 

Let $L(\gamma,\lambda):=\min\{\ell(\gamma),\ell(\lambda)\}$ for any $\lambda,\gamma\in\mathcal{P}$.  We denote the equivalence class of each partition $\lambda\in\mathcal{P}$ by
$$ [\lambda] = \left\{ \gamma\in\mathcal{P}:\gamma_j\equiv\lambda_j\hspace{-.25cm}\pmod{k}\text{ for }1\leq j\leq L(\gamma,\lambda);\,\gamma_j\equiv\lambda_j\equiv0\hspace{-.25cm}\pmod{k}\text{ for }j>L(\gamma,\lambda)\right\}.$$ We define the set $\mathcal{P}/_{\sim_k}:=\{[\lambda]:\lambda\in\mathcal{P}\}$.

\begin{example}
The equivalence class of the partition $\lambda=(6,1)$ modulo $k=4$ contains the minimal-size partition $(2,1)$ as well as the partitions $(6,5)$, $(6,5,4)$, and $(22,17,16,8,4,4)$.
\end{example}

\begin{definition}
We define the binary operation \textit{component-wise addition} on $\mathcal{P}/_{\sim_k}$ as follows.  Let $[\lambda]=[\left(\lambda_1,\lambda_2,\dots,\lambda_\ell\right)]$ and $[\pi]=[\left(\pi_1,\pi_2,\dots,\pi_i\right)]$ be equivalence classes of partitions in $\mathcal{P}_{\sim_k}$, and assume $\ell\geq i$. Then their component-wise sum is $$[\lambda]+[\pi]:=\left[\left(\lambda_1+\pi_1,\lambda_2+\pi_2,\dots,\lambda_i+\pi_i,\lambda_{i+1},\dots,\lambda_\ell\right)\right].$$
\end{definition}

In other words, the component-wise sum of two equivalence classes $[\lambda],[\pi]$ of partitions is defined as the equivalence class of the partition which is the result of component-wise adding the respective parts of the two representative partitions $\lambda,\pi$.  We show here that component-wise addition is a well-defined operation on pairs of partition equivalence classes in $\mathcal{P}_{\sim_k}$. Recall that we use the convention that for any partition $\mu=\left(\mu_1,\mu_2,\dots,\mu_m\right)\in\mathcal{P}$, we set $\mu_j=0$ for all $j>m$. Let $\lambda=\left(\lambda_1,\lambda_2,\dots,\lambda_\ell\right),\pi=\left(\pi_1,\pi_2,\dots,\pi_\ell\right)\in\mathcal{P}$ with $\ell(\lambda)\geq\ell(\pi)$. Let $[\lambda],[\pi]\in\mathcal{P}_{\sim_k}$, and let $\gamma=\left(\gamma_1,\gamma_2,\dots,\gamma_\ell\right),\delta=\left(\delta_1,\delta_2,\dots,\delta_\ell\right)\in\mathcal{P}$ such that $\gamma\in[\lambda]$, $\delta\in[\pi]$; we assume without loss of generality that $\ell(\lambda)\geq\ell(\gamma),\ell(\delta)$ as well.  Then $\gamma$ and $\delta$ satisfy the congruences $\gamma_j\equiv\lambda_j\pmod{k}$ and $\delta_j\equiv\pi_j\pmod{k}$ for all $j\geq1$, so we have that
\begin{align*}
[\lambda]+[\pi]&=\left[\left(\lambda_1+\pi_1,\lambda_2+\pi_2,\dots,\lambda_\ell+\pi_\ell\right)\right]\\
&=\left[\left(\gamma_1+\delta_1,\gamma_2+\delta_2,\dots,\gamma_\ell+\delta_\ell\right)\right]\\
&=[\gamma]+[\delta].
\end{align*}

\begin{theorem}\label{group_eq_classes}
Let $k$ be a positive integer. The set $\mathcal{P}/_{\sim_k}$ forms an abelian group under component-wise addition.
\end{theorem}

\begin{proof}
Let $k$ be a positive integer, and let $\lambda=\left(\lambda_1,\lambda_2,\dots,\lambda_\ell\right)$, $\pi=\left(\pi_1,\pi_2,\dots,\pi_i\right)$, and $\tau=\left(\tau_1,\tau_2,\dots,\tau_t\right)$ be partitions in $\mathcal{P}$ with $\ell\geq i\geq t$.  It is clear by definition that $\mathcal{P}/_{\sim_k}$ is closed under the operation of component-wise addition.

The equivalence class $[\emptyset]\in\mathcal{P}/_{\sim_k}$ is the identity element, since $[\lambda]+[\emptyset]=[\emptyset]+[\lambda]=[\lambda]$. Next, consider any partition $\gamma\in\mathcal{P}$ such that $\gamma_j:=r_jk-\lambda_j$ for some positive integer $r_j$. The equivalence class $[\gamma]\in\mathcal{P}/_{\sim_k}$ is the inverse element of $[\lambda]$, since
\begin{align*}
[\lambda]+[\gamma]&=\left[\left(\lambda_1+\left(r_1k-\lambda_1\right),\lambda_2+\left(r_2-\lambda_2\right),\dots,\lambda_\ell-\left(r_\ell k-\lambda_\ell\right)\right)\right]\\
&=\left[\left(r_1k,r_2k,\dots,r_\ell k\right)\right]\\
&=[\emptyset].
\end{align*}
Now, it is straightforward to show that component-wise addition is both associative and commutative, since this operation does not depend on the choice of equivalence class representative and simply consists of integer addition in each component of the representative partitions.  We briefly show associativity and commutativity here: we have that
\begin{align*}
([\lambda]+[\pi])+[\tau]&=\left[\left(\lambda_1+\pi_1,\dots,\lambda_\ell+\pi_\ell\right)\right]+[\tau]\\
&=\left[\left(\lambda_1+\pi_1+\tau_1,\dots,\lambda_\ell+\pi_\ell+\tau_\ell\right)\right]\\
&=[\lambda]+\left[\left(\pi_1+\tau_1,\dots,\pi_\ell+\tau_\ell\right)\right]\\
&=[\lambda]+([\pi]+[\tau])
\end{align*}
and
\begin{align*}
[\lambda]+[\pi]&=\left[\left(\lambda_1+\pi_1,\dots,\lambda_\ell+\pi_\ell\right)\right]\\
&=\left[\left(\pi_1+\lambda_1,\dots,\pi_\ell+\lambda_\ell\right)\right]\\
&=[\pi]+[\lambda].
\end{align*}
Thus $\mathcal{P}/_{\sim_k}$ is an abelian group under component-wise addition.
\end{proof}

We now define the operation of component-wise addition on $\mathcal{P}$, but we must incorporate a type of reduction to obtain an abelian group structure on partitions under this operation. We will show in the proof of Theorem \ref{group_structures}, Part 2, that the reduction we use is equivalent to considering only the partition of minimal size in each equivalence class of $\mathcal{P}/_{\sim_k}$.

\begin{definition}\label{componentwise_reduction_defn}
For any positive integer $k$, we define the unary operation \textit{component-wise reduction modulo $k$} on any partition $\lambda\in\mathcal{P}$ as follows.  First, replace the part $\lambda_{\ell(\lambda)}$ with the smallest non-negative integer $\lambda_{\ell(\lambda)}'$ which is congruent to $\lambda_{\ell(\lambda)}$ modulo $k$; namely, we have that $\lambda_{\ell(\lambda)}':=\lambda_{\ell(\lambda)}-g_0k$, where $g_0:=\lfloor\lambda_{\ell(\lambda)}/k\rfloor$.  Then, for each $1\leq j<\ell(\lambda)$, replace the part $\lambda_j$ by the positive integer $\lambda_j':=\lambda_j-g_jk$, where $g_j$ is defined backward recursively by $$g_j:=\left\lfloor\frac{\lambda_j-\lambda_{j+1}'}{k}\right\rfloor.$$  The new partition $\lambda':=\left(\lambda_1',\lambda_2',\dots,\lambda_\ell'\right)$ is the component-wise reduction modulo $k$ of $\lambda$. We call a partition $\lambda\in\mathcal{P}$ \textit{component-wise reduced modulo $k$} if the difference between any two consecutive parts of $\lambda$ is less than $k$ and the smallest part $\lambda_{\ell(\lambda)}$ is less than $k$.
\end{definition}

Note that component-wise reduction modulo $k$ reduces a partition $\lambda\in\mathcal{P}$ to the partition $\lambda'$ of minimal size such that each part of $\lambda'$ is congruent modulo $k$ to the corresponding part of $\lambda$. This minimal size is achieved by subtracting the largest possible multiple of $k$ from each part so that it remains larger than the next part, and therefore $\lambda'$ remains a partition.  In addition, if the $s$ smallest parts of $\lambda$ are all divisible by $k$, for $0\leq s\leq\ell(\lambda)$, then component-wise reduction modulo $k$ on $\lambda$ results in all of the $s$ smallest parts being replaced by 0, and these parts will be then omitted from $\lambda'$.  As a result, when a partition is component-wise reduced modulo $k$, its length decreases by the number of its smallest parts which are all simultaneously divisible by $k$.

\begin{example}
Let $k=4$ and $\lambda=(24,18,13,9,7,2)\in\mathcal{P}$.  We perform component-wise reduction modulo $4$ on $\lambda$ as follows:
\begin{align*}
(24,18,13,9,7,2)&\to(24,18,13,9,7,2-0\cdot4)\\
&\to(24,18,13,9,7-1\cdot4,2)\\
&\to(24,18,13,9-1\cdot4,3,2)\\
&\to(24,18,13-2\cdot4,5,3,2)\\
&\to(24,18-3\cdot4,5,5,3,2)\\
&\to(24-4\cdot4,6,5,5,3,2)=(8,6,5,5,3,2).
\end{align*}
Observe that the sequence $g_j$, $0\leq j<\ell(\lambda)$, in this example is given by
\begin{align*}
g_0&=\left\lfloor\frac{2}{4}\right\rfloor=0,\\
g_1&=\left\lfloor\frac{7-2}{4}\right\rfloor=1,\\
g_2&=\left\lfloor\frac{9-3}{4}\right\rfloor=1,\\
g_3&=\left\lfloor\frac{13-5}{4}\right\rfloor=2,\\
g_4&=\left\lfloor\frac{18-5}{4}\right\rfloor=3,\\
g_5&=\left\lfloor\frac{24-6}{4}\right\rfloor=4,
\end{align*}
and that the resulting partition $\lambda'=(8,6,5,5,3,2)\in\mathcal{P}_k^+$ is the partition of smallest size in $\mathcal{P}$ whose parts are congruent modulo $4$ to the corresponding parts of $\lambda=(24,18,13,9,7,2)$.
\end{example}

\begin{example}
Let $k=3$ and $\lambda=(8,7,6,3,3,3)\in\mathcal{P}$.  Component-wise reduction modulo $3$ on $\lambda$ yields $$(8,7,6,3,3,3)\to(2,1).$$ Note in particular that the reduced partition $\lambda'=(2,1)$ has length $\ell(\lambda')=\ell(\lambda)-4=2$, since the number of smallest parts of $\lambda$ which are simultaneously divisible by $3$ is $4$.
\end{example}

See Figure \ref{equivalenceClassExample} for a general diagram of component-wise reduction modulo $k$ on a family of length-4 partitions whose parts are congruent to $3,2,1,1$ (respectively) modulo $k$.

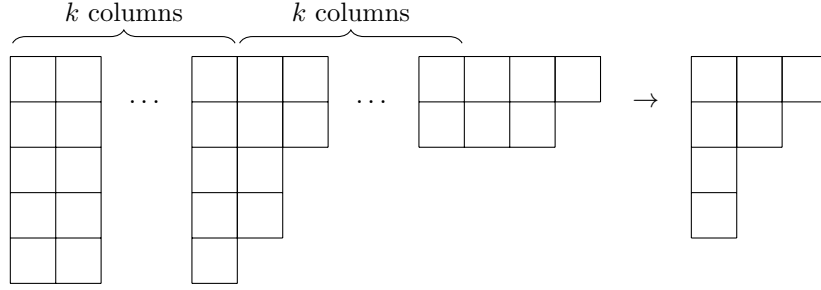
\begin{figure}[h]
    \centering
    \begin{tikzpicture}[scale=0.6,step=1cm]
    \usetikzlibrary{calc}
    \draw (0,0) grid (2,-5);
    \node at (3,-1) {$\dots$};
    \draw (4,0) grid (5,-5);
    \draw (5,0) grid (6,-4);
    \draw (6,0) grid (7,-2);
    \node at (8,-1) {$\dots$};
    \draw (9,0) grid (12,-2);
    \draw (12,0) grid (13,-1);
    \node at (14,-1) {$\to$};
    \draw (15,0) grid (16,-4);
    \draw (16,0) grid (17,-2);
    \draw (17,0) grid (18,-1);
    \node (A) at (0,0) {};
    \node (B) at (5,0) {};
    \node (C) at (10,0) {};
    \draw[decoration={brace, raise=5pt, amplitude=5pt}, decorate] 
        ($(A)+(0.05,0)$) -- ($(B)+(-0.05,0)$) node[midway, above=10pt] {$k$ columns};
   \draw[decoration={brace, raise=5pt, amplitude=5pt}, decorate]
        ($(B)+(0.05,0)$) -- ($(C)+(-0.05,0)$) node[midway, above=10pt] {$k$ columns};
    \end{tikzpicture}
    \caption{Component-wise reduction modulo $k$: $(2k +3, 2k + 2, k+1, k+1,k) \to (3, 2, 1, 1)$}
    \label{equivalenceClassExample}
\end{figure}

\begin{definition}
For any positive integer $k$, we define the binary operation \textit{component-wise addition and reduction modulo $k$}, denoted $+_k$, on $\mathcal{P}$ as follows. Let $\lambda=\left(\lambda_1,\lambda_2,\dots,\lambda_\ell\right)$ and $\pi=\left(\pi_1,\pi_2,\dots,\pi_i\right)$ be any two partitions in $\mathcal{P}$ with $\ell\geq i$.  Define $\lambda+_k\pi$ by first component-wise adding the parts of $\lambda$ and $\pi$: $\left(\lambda_1+\pi_1,\lambda_2+\pi_2,\dots,\lambda_i+\pi_i,\lambda_{i+1},\dots,\lambda_\ell\right)$, and then component-wise reducing the result modulo $k$.
\end{definition}

\begin{example}
For $k=4$, we have $(4, 3, 3, 1)+_4(3, 3, 2) = (3, 2, 1, 1)$. Figure \ref{additionExampleMod4} shows the Young diagrams involved in this example of component-wise addition and reduction. 
    \begin{figure}[h]
        \centering
        \begin{tikzpicture}[scale=0.6,step=1cm]
    \draw (0,0) grid (3,-3);
    \draw (3,0) grid (4,-1);
    \draw (0,-3) grid (1,-4);
    \node at (5,-1) {$+_4$};
    \draw (6,0) grid (9,-2);
    \draw (6,-2) grid (8,-3);
    \node at (10,-1) {$=$};
    \draw (11,0) grid (14,-1);
    \draw (11,-1) grid (13,-2);
    \draw (11,-2) grid (12,-4);
    \end{tikzpicture}
        \caption{Component-wise addition and reduction modulo $4$: $(4, 3, 3, 1)+_4(3, 3, 2) = (3,2,1,1)$}
        \label{additionExampleMod4}
    \end{figure}
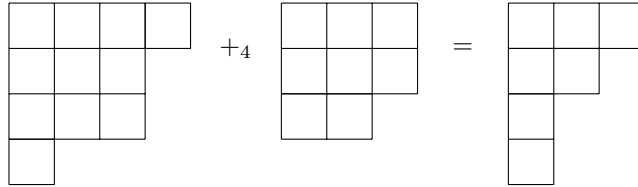
    
    Note that the component-wise sum of $(4,3,3,1)$ and $(3,3,2)$ in $\mathcal{P}$ yields the partition $(7,6,5,1)$, which then component-wise reduces modulo $4$ to $(3,2,1,1)\in\mathcal{P}_4^+$.
\end{example}

Now, we can prove Theorem \ref{group_structures}, Part 2.

\begin{proof}[Proof of Theorem \ref{group_structures}, Part 2]
We prove that the set $\mathcal{P}_k^+$ with the operation component-wise addition and reduction modulo $k$ is an abelian group by proving that it is isomorphic to the abelian group $\mathcal{P}/_{\sim_k}$ under component-wise addition.

Let $\lambda\in\mathcal{P}$ be any partition. Define the map $\varphi$ on $\mathcal{P}/_{\sim_k}$ by $\varphi([\lambda]):=\gamma$, where $\gamma$ is the result of component-wise reduction modulo $k$ on $\lambda$.  Then $\gamma=\left(\gamma_1,\gamma_2,\dots,\gamma_\ell\right)$, where
\begin{align}
\gamma_\ell:=\lambda_\ell-\left\lfloor\frac{\lambda_\ell}{k}\right\rfloor k\label{gamma_ell_defn}
\end{align}
and, for each $1\leq j<\ell$, we define $\gamma_j$ backward recursively by
\begin{align}
\gamma_j:=\lambda_j-\left\lfloor\frac{\lambda_j-\gamma_{j+1}}{k}\right\rfloor k.\label{gamma_j_defn}
\end{align}
Let $\pi\in\mathcal{P}$ be any partition, and suppose that $\varphi([\pi])=\nu:=\left(\nu_1,\nu_2,\dots,\nu_\ell\right)$. If $\nu=\gamma$, then we have that
\begin{align}
\nu_\ell:=\pi_\ell-\left\lfloor\frac{\pi_\ell}{k}\right\rfloor k=\lambda_\ell-\left\lfloor\frac{\lambda_\ell}{k}\right\rfloor k=\gamma_\ell\label{ell}
\end{align}
and, for each $1\leq j<\ell$, we have
\begin{align}
\nu_j:=\pi_j-\left\lfloor\frac{\pi_j-\nu_{j+1}}{k}\right\rfloor k=\lambda_j-\left\lfloor\frac{\lambda_j-\gamma_{j+1}}{k}\right\rfloor k=\gamma_j.\label{other}
\end{align}
From \eqref{ell} and \eqref{other}, we see that if $\nu=\gamma$, then $\pi_j\equiv\lambda_j\pmod{k}$ for all $1\leq j\leq\ell$, and therefore $\pi$ and $\lambda$ are component-wise congruent modulo $k$.  Thus, we have shown that $\varphi([\lambda])=\varphi([\pi])$ implies $[\pi]=[\lambda]$, and so $\varphi$ is one-to-one.

Now, observe that the image of each $[\lambda]\in\mathcal{P}/_{\sim_k}$ is a component-wise reduced modulo $k$ partition, so $\varphi$ maps to $\mathcal{P}_k^+$.  Let $\mu\in\mathcal{P}_k^+$.  Then we have that $\varphi([\mu])=\mu$, since $\mu$ is already component-wise reduced modulo $k$.  Thus, we have shown that $\varphi:\mathcal{P}/_{\sim_k}\to\mathcal{P}_k^+$ is onto.

Finally, we show that $\varphi([\lambda]+[\pi])=\varphi([\lambda])+_k\varphi([\pi])$ for any $[\lambda],[\pi]\in\mathcal{P}/_{\sim_k}$. On the left, we have that
\begin{align*}
\varphi([\lambda]+[\pi])&=\varphi\left(\left[\left(\lambda_1+\pi_1,\lambda_2+\pi_2,\dots,\lambda_\ell+\pi_\ell\right)\right]\right)=\rho:=\left(\rho_1,\rho_2,\dots,\rho_\ell\right),
\end{align*}
where $$\rho_\ell:=\lambda_\ell+\pi_\ell-\left\lfloor\frac{\lambda_\ell+\pi_\ell}{k}\right\rfloor k$$ and $$\rho_j:=\lambda_j+\pi_j-\left\lfloor\frac{\lambda_j+\pi_j-\rho_{j+1}}{k}\right\rfloor k$$ for each $1\leq j<\ell$. On the other hand, we have that $\varphi([\lambda])+_k\varphi([\pi])=\gamma+_k\nu$, where the parts of $\gamma$ are defined in \eqref{gamma_ell_defn} and \eqref{gamma_j_defn}, and the parts of $\nu$, although now no longer necessarily equal to the respective parts of $\gamma$, are defined on the left sides of \eqref{ell} and \eqref{other}.  Then we have that $\varphi([\lambda])+_k\varphi([\pi])=\zeta:=\left(\zeta_1,\zeta_2,\dots,\zeta_\ell\right)$, where $$\zeta_\ell:=\gamma_\ell+\nu_\ell-\left\lfloor\frac{\gamma_\ell+\nu_\ell}{k}\right\rfloor k$$ and $$\zeta_j:=\gamma_j+\nu_j-\left\lfloor\frac{\gamma_j+\nu_j-\zeta_{j+1}}{k}\right\rfloor k$$ for each $1\leq j<\ell$. We see that $\rho_j=\zeta_j$ for all $1\leq j\leq\ell$, because both $\rho_j$ and $\zeta_j$ are component-wise reduced modulo $k$ and $\lambda_j+\pi_j\equiv\gamma_j+\nu_j\pmod{k}$. Thus, we have shown that $\rho=\zeta$, and therefore $\varphi$ is an isomorphism.
\end{proof}

In fact, for any positive integer $k$, the groups $\mathcal{P}_{\cup_k}$ and $\mathcal{P}_k^+$ are isomorphic by conjugation.

\begin{theorem}\label{isom_2}
Let $k$ be a positive integer. The group $\mathcal{P}_{\cup_k}$ is isomorphic to the group $\mathcal{P}_k^+$.
\end{theorem}

\begin{proof}
Consider the map $\theta:\mathcal{P}\to\mathcal{P}$ given by conjugation, which is defined by reflecting the Young diagram of each partition across the main diagonal. It is well known that $\theta$ is a bijection from $\mathcal{P}$ to $\mathcal{P}$. For convenience, we denote the frequency notation of any partition $\lambda=\left(\lambda_1,\lambda_2,\dots,\lambda_\ell\right)\in\mathcal{P}$ by $\lambda=\left\langle 1^{m_\lambda(1)},2^{m_\lambda(2)},\dots,r^{m_\lambda(r)}\right\rangle$, where $r$ is the largest part of $\lambda$, and we recall that $m_\lambda(i)$ is the number of times $i$ occurs as a part in $\lambda$. Observe that the length $\ell$ of $\lambda$ satisfies $\ell=m_\lambda(1)+m_\lambda(2)+\cdots+m_\lambda(r)$. Then we have that
\begin{align*}
\theta(\lambda)=\left(m_\lambda(1)+m_\lambda(2)+\cdots+m_\lambda(r),m_\lambda(2)+\cdots+m_\lambda(r),\dots,m_\lambda(r)\right),
\end{align*}
and, for any partition $\pi=(\pi_1,\pi_2,\dots,\pi_s)\in\mathcal{P}$, we have
\begin{align*}
\theta^{-1}(\pi)=\left\langle 1^{\pi_1-\pi_2},2^{\pi_2-\pi_3},\dots,(s-1)^{\pi_{s-1}-\pi_s},s^{\pi_s}\right\rangle.
\end{align*}
We see that $\theta(\lambda)\in\mathcal{P}_k^+$ for each $\lambda\in\mathcal{P}_{\cup_k}$, since each multiplicity $m_\lambda(i)$ in $\lambda$ is less than $k$, and therefore each pair of consecutive parts in $\theta(\lambda)$ differs by at most $k-1$ and the smallest part of $\theta(\lambda)$ is $m_\lambda(r)<k$.

Now, we restrict $\theta$ to $\mathcal{P}_{\cup_k}$ and show that $\theta$ is a bijection from $\mathcal{P}_{\cup_k}$ to $\mathcal{P}_k^+$. Let $\lambda,\pi\in\mathcal{P}_{\cup_k}$ such that $\theta(\lambda)=\theta(\pi)$. Then $\lambda$ and $\pi$ are partitions into parts with multiplicity less than $k$, and each part in $\theta(\lambda)$ is congruent modulo $k$ to the corresponding part in $\theta(\pi)$. If $\theta(\lambda)=\left(\gamma_1,\gamma_2,\dots,\gamma_r\right)$, then $\theta(\pi)=\left(\gamma_1+c_1k,\gamma_2+c_2k,\dots,\gamma_r+c_r k\right)$ for $k_i\in\mathbb{Z}$, $1\leq i\leq r$. We see that
\begin{align*}
\lambda=\theta^{-1}(\theta(\lambda))=\left\langle1^{\gamma_1-\gamma_2},2^{\gamma_2-\gamma_3},\dots,r^{\gamma_r}\right\rangle
\end{align*}
and
\begin{align*}
\pi=\theta^{-1}(\theta(\pi))&=\left\langle1^{\gamma_1+c_1k-\gamma_2-c_2k},2^{\gamma_2+c_2k-\gamma_3-c_3k},\dots,r^{\gamma_r+c_r k}\right\rangle\\
&=\left\langle1^{\gamma_1-\gamma_2+(c_1-c_2)k},2^{\gamma_2-\gamma_3+(c_2-c_3)k},\dots,r^{\gamma_r+c_r k}\right\rangle.
\end{align*}
Thus, since the multiplicity of each part size differs between $\lambda$ and $\pi$ by multiples of $k$, and both $\lambda$ and $\pi$ are multiplicity-reduced modulo $k$, we see that $\lambda=\pi$ and $\theta$ is one-to-one. For $\pi=(\pi_1,\pi_2,\dots,\pi_s)\in\mathcal{P}_k^+$, we have that the difference $\pi_i-\pi_{i+1}$, $1\leq i<s$, between consecutive parts in $\pi$ is less than $k$ and the smallest part $\pi_s$ of $\pi$ is less than $k$. Then the partition $\delta=\left\langle 1^{\pi_1-\pi_2},2^{\pi_2-\pi_3},\dots,(s-1)^{\pi_{s-1}-\pi_s},s^{\pi_s}\right\rangle$ is in $\mathcal{P}_{\cup_k}$ and $\theta(\delta)=\pi$, so $\theta$ maps onto $\mathcal{P}_k^+$. We have shown that $\theta:\mathcal{P}_{\cup_k}\to\mathcal{P}_k^+$ is a bijection.

Finally, we show that $\theta:\mathcal{P}_{\cup_k}\to\mathcal{P}_k^+$ is a homomorphism. For the partitions $\lambda=\left\langle 1^{m_\lambda(1)},2^{m_\lambda(2)},\dots,r^{m_\lambda(r)}\right\rangle$ and $\pi=\left\langle 1^{m_\pi(1)},2^{m_\pi(2)},\dots,s^{m_\pi(s)}\right\rangle$ in $\mathcal{P}_{\cup_k}$, assuming without loss of generality that $r\geq s$, we have that
\begin{align*}
\lambda\cup_k\pi=\left\langle1^{m_1'},2^{m_2'},\dots,s^{m_s'},(s+1)^{m_{s+1}'},\dots,r^{m_r'}\right\rangle,
\end{align*}
where $m_i'\equiv m_\lambda(i)+m_\pi(i)\pmod{k}$ for all $1\leq i\leq s$, $m_i'\equiv m_\lambda(i)\pmod{k}$ for all $s+1\leq i\leq r$, and $m_i'<k$ for all $1\leq i\leq r$. Then
\begin{align*}
\theta(\lambda\cup_k\pi)&=(m_1'+m_2'+\cdots+m_r',m_2'+\cdots+m_r',\dots,m_r'),\\
&=\left(m_\lambda(1)+m_\lambda(2)+\cdots+m_\lambda(r),m_\lambda(2)+\cdots+m_\lambda(r),\dots,m_\lambda(r)\right)\\
&\qquad+_k\left(m_\pi(1)+m_\pi(2)+\cdots+m_\pi(s),m_\pi(2)+\cdots+m_\pi(s),\dots,m_\pi(s)\right)\\
&=\theta(\lambda)+_k\theta(\pi).
\end{align*}
We note that all partitions involved in this equation are reduced modulo $k$ by the definitions of $\mathcal{P}_{\cup_k}$, $\cup_k$, and $+_k$. Thus $\theta$ is an isomorphism.
\end{proof}

Theorem \ref{isom_2} shows that the set of partitions whose parts occur less than $k$ times, under the operation of concatenation and reduction modulo $k$, is isomorphic by conjugation to the set of partitions with smallest part less than $k$ in which consecutive parts differ by less than $k$, under the operation of component-wise addition and reduction modulo $k$.  In other words, concatenation and reduction modulo $k$ on multiplicity-reduced modulo $k$ partitions is the conjugate of component-wise addition and reduction modulo $k$ on component-wise reduced modulo $k$ partitions.

For any positive integer $k$, any fixed partition $\lambda\in\mathcal{P}_k^+$, and any fixed non-negative integer $m$, let $p_{\lambda\bmod k}(m)$ denote the number of partitions in $\mathcal{P}$ of size $|\lambda|+mk$ which component-wise reduce modulo $k$ to the same partition $\lambda$.  In other words, $p_{\lambda\bmod{k}}(m)$ counts the number of partitions in $\mathcal{P}$ whose parts are component-wise congruent modulo $k$ to the partition $\lambda\in\mathcal{P}_k^+$ (equivalently, the number of partitions in the equivalence class $[\lambda]\in\mathcal{P}/_{\sim_k}$) and whose size is $km$ larger than $|\lambda|$.  This function can be viewed as counting the number of ways a component-wise reduced modulo $k$ partition can have $m$ additions of size $k$ to any of its parts, where more than one addition of size $k$ is allowed in a single part when $m\geq2$. An interesting observation relates the size of the restricted partition function $p_{\lambda\bmod k}(m)$ for any $m\geq0$ and the size of the ordinary partition function $p(m)$, which counts the number of partitions in $\mathcal{P}$ of size $m$.  In particular, we observe the following relationship.

\begin{theorem}
For any positive integer $k$, any non-negative integer $m$, and any partition $\lambda\in\mathcal{P}_k^+$, we have that $p_{\lambda\bmod k}(m)=p(m)$.
\end{theorem}

\begin{proof}
After establishing a few necessary properties of the partitions counted by $p_{\lambda\bmod{k}}(m)$, we will prove that $p_{\lambda\bmod{k}}(m)=p(m)$ by explicitly constructing a bijection between the two sets of partitions that these functions count.  Fix $\lambda=\left(\lambda_1,\lambda_2,\dots,\lambda_\ell\right)\in\mathcal{P}_k^+$ and $m\geq0$.  Then we have that $|\lambda|=\sum_i\lambda_i$ with $\lambda_\ell<k$ and $\lambda_i-\lambda_{i+1}<k$ for all $1\leq i<\ell$.  Each partition $\gamma=\left(\gamma_1,\gamma_2,\dots,\gamma_\ell\right)\in\mathcal{P}$ counted by $p_{\lambda\bmod k}(m)$ has size $\left(\sum_i\lambda_i\right)+mk$ and satisfies $\gamma_i\equiv\lambda_i$ for all $1\leq i\leq\ell$.  Suppose that $\gamma_i=\lambda_i+h_ik$ with $h_i\geq0$ for each $1\leq i\leq\ell$.  Then we have that $\sum_ih_i=m$.  We will show that the number of possible component-wise additions by the $\ell$-tuple $\left(h_1,h_2,\dots,h_\ell\right)$ is equal to the number of partitions of $m$.

First, if $m=0$, then $h_i=0$ for all $1\leq i\leq\ell$ and $\gamma=\lambda$ is the only partition counted by $p_{\lambda\bmod{k}}(0)$.  On the other hand, $p(0)=1$ as well.  Next, if $m=1$, then $h_i=1$ for exactly one $i$ with $1\leq i\leq\ell$, and $h_j=0$ for every $j\neq i$ with $1\leq j\leq\ell$.  In fact, the only possible value of $i$ for which we can have $h_i=1$ is $i=1$, for the following reason: if $h_i=1$ for $i\neq1$, then $\gamma=\left(\lambda_1,\dots,\lambda_{i-1},\lambda_i+k,\lambda_{i+1},\dots,\lambda_\ell\right)$ is not a partition, because $\lambda_{i-1}-\lambda_i<k$. So we see that $p_{\lambda\bmod{k}}(1)=p(1)=1$.  For the same reason, we must have, for any $m\geq0$, that $h_i\geq h_{i+1}$ for all $1\leq i<\ell$ in order for $\gamma$ to be a partition.

Define the map $\vartheta:\gamma\mapsto\left(h_1,h_2,\dots,h_\ell\right)$, omitting all trailing parts $h_i=0$. Since we have that $\sum_ih_i=m$ and $h_i\geq h_{i+1}$ for all $1\leq i<\ell$, we see that $\vartheta$ is a map from the set of partitions counted by $p_{\lambda\bmod{k}}(m)$ to the set of all ordinary partitions of size $m$.  Since each choice of $\gamma=\left(\gamma_1,\gamma_2,\dots,\gamma_\ell\right)\in[\lambda]$ uniquely determines the partition $\left(h_1,h_2,\dots,h_\ell\right)$ by $h_i:=\left(\gamma_i-\lambda_i\right)/k$ for each $1\leq i\leq\ell$, we have that $\vartheta$ is one-to-one.  We see that $\vartheta$ is also onto because, for any partition $\left(h_1,h_2,\dots,h_\ell\right)\in\mathcal{P}$, we have that $\left(\lambda_1+h_1k,\lambda_2+h_2k,\dots,\lambda_\ell+h_\ell k\right)\in[\lambda]$ and $\vartheta:\left(\lambda_1+h_1k,\lambda_2+h_2k,\dots,\lambda_\ell+h_\ell k\right)\mapsto\left(h_1,h_2,\dots,h_\ell\right)$.  Thus, $\vartheta$ is a bijection, so $p_{\lambda\bmod k}(m)=p(m)$.
\end{proof}

\begin{example}
Let $k=5$ and $\lambda=(8,6,4,3)\in\mathcal{P}_k^+$.  Then $|\lambda|=21$, and each partition $\gamma$ in $\mathcal{P}$ counted by $p_{\lambda\bmod{5}}(m)$, $m\geq0$, is of the form $\gamma=\left(8+5h_1,6+5h_2,4+5h_3,3+5h_4\right)$ for some partition $\left(h_1,h_2,h_3,h_4\right)\in\mathcal{P}$ of size $h_1+h_2+h_3+h_4=m$.  For example, if $m=4$, then the $p(4)=5$ possible partitions $\gamma$ which component-wise reduce modulo 5 to $\lambda=(8,6,4,3)$ and have size $21+4\cdot5=41$ are:
\begin{align*}
\gamma&=(8+4\cdot5,6,4,3)=(28,6,4,3);\\
\gamma&=(8+3\cdot5,6+1\cdot5,4,3)=(23,11,4,3);\\
\gamma&=(8+2\cdot5,6+2\cdot5,4,3)=(18,16,4,3);\\
\gamma&=(8+2\cdot5,6+1\cdot5,4+1\cdot5,3)=(18,11,9,3);\\
\gamma&=(8+1\cdot5,6+1\cdot5,4+1\cdot5,3+1\cdot5)=(13,11,9,8),
\end{align*}
where each possibility for $\gamma$ results from component-wise adding one ordinary partition of 4, with each part multiplied by 5, to $\lambda$. Namely, we component-wise add the following partitions to $\lambda=(8,6,4,3)$ to get the five possibilities for $\gamma$ above:
\begin{align*}
(4\cdot5)&=(20),\\
(3\cdot5,1\cdot5)&=(15,5),\\
(2\cdot5,2\cdot5)&=(10,10),\\
(2\cdot5,1\cdot5,1\cdot5)&=(10,5,5),\\
(1\cdot5,1\cdot5,1\cdot5,1\cdot5)&=(5,5,5,5).
\end{align*}
\end{example}

\subsection{Group Structure on $\mathcal{P}_{p,r}^{(0)}$}
In this subsection, we restrict the set $\mathcal{P}$ of partitions appropriately to obtain an abelian group structure under the operation of component-wise multiplication.

\begin{definition}
Let $k$ be a positive integer, and let $\lambda,\pi\in\mathcal{P}$ such that $\lambda=\left(\lambda_1,\lambda_2,\dots,\lambda_\ell\right)$ and $\pi=\left(\pi_1,\pi_2,\dots,\pi_\ell\right)$. We define the relation \textit{component-wise zero-congruence modulo $k$}, denoted $\sim_k'$ as follows.
\begin{itemize}
\item We say that $\lambda\sim_k'\emptyset$ if and only if $\lambda$ has some part divisible by $k$ or is $\emptyset$.
\item If both $\lambda$ and $\pi$ have no part divisible by $k$, then we say that $\lambda\sim_k'\pi$ if and only if $\lambda_j\equiv\pi_j\pmod{k}$ for all $1\leq j\leq\ell$.
\end{itemize}
If $\lambda\sim_k'\pi$, then we say that $\lambda$ and $\pi$ are \textit{component-wise zero-congruent modulo $k$}.
\end{definition}

Note that two partitions of different length and with no parts divisible by $k$ cannot be component-wise zero-congruent modulo $k$.  Next, we state that this relation is, in fact, an equivalence relation on partitions.  We omit the proof, because it is identical to the proof of Theorem \ref{additionequivalence} for partitions with no parts divisible by $k$.

\begin{theorem} \label{multiplicationequivalence}
    For any integer $k$, component-wise zero-congruence modulo $k$ is an equivalence relation on $\mathcal{P}$.
\end{theorem}

\begin{example}\label{mult_eq_class_ex}
The equivalence class of the length-3 partition $(6,5,1)$ modulo $k=4$ contains the minimal-size partition $(2,1,1)$ as well as the partitions $(6,1,1)$ and $(14,13,5)$. However, the equivalence class of the partition $(6,5,4)$ modulo $k=4$ contains the empty partition $\emptyset$ and all partitions of length 3 with any part divisible by 4.
\end{example}

In order to obtain a group structure on equivalence classes of partitions under component-wise zero-congruence modulo $k$, we must restrict to a fixed partition length to preserve uniqueness of inverse partitions. For a non-negative integer $r$, we denote the set of equivalence classes of length-$r$ partitions under component-wise zero-congruence modulo $k$ by $\mathcal{P}_r/_{\sim_k'}$.

\begin{definition}
We define the binary operation \textit{component-wise multiplication} on $\mathcal{P}_r/_{\sim_k'}$ as follows. Let $[\lambda]=\left[\left(\lambda_1,\lambda_2,\dots,\lambda_r\right)\right]$ and $[\pi]=\left[\left(\pi_1,\pi_2,\dots,\pi_r\right)\right]$ be equivalence classes of partitions in $\mathcal{P}_r/_{\sim_k'}$. Then their component-wise product is $$[\lambda]\cdot[\pi]:=\left[\left(\lambda_1\pi_1,\lambda_2\pi_2,\dots,\lambda_r\pi_r\right)\right].$$
\end{definition}

In other words, the component-wise product of two equivalence classes $[\lambda],[\pi]$ of partitions is defined as the equivalence class of the partition which is the result of component-wise multiplying the respective parts of the two representative partitions $\lambda,\pi$. Similarly to component-wise addition, it is straightforward to see that component-wise multiplication is a well-defined operation on pairs of partition equivalence classes in $\mathcal{P}_r/_{\sim_k'}$.

\begin{theorem}
Let $p$ be a prime, and let $r$ be a non-negative integer.  The set $\mathcal{P}_r/_{\sim_p'}$ forms an abelian group under component-wise multiplication.
\end{theorem}

\begin{proof}
Let $\lambda=\left(\lambda_1,\lambda_2,\dots,\lambda_r\right)$, $\pi=\left(\pi_1,\pi_2,\dots,\pi_r\right)$, and $\tau=\left(\tau_1,\tau_2,\dots,\tau_r\right)$ be partitions. It is clear by definition that $\mathcal{P}_r/_{\sim_p'}$ is closed under the operation of component-wise multiplication.

The equivalence class $[(\underbrace{1,\dots,1)}_{r\text{ parts}})]\in\mathcal{P}_r/_{\sim_p'}$ is the identity element, since clearly $$[\lambda]\cdot[(\underbrace{1,\dots,1}_{r\text{ times}})]=[(\underbrace{1,\dots,1}_{r\text{ times}})]\cdot[\lambda]=[\lambda].$$ Next, consider any partition $\gamma$ of length $r$ whose parts $\gamma_j$ satisfy $\lambda_j\gamma_j=\gamma_j\lambda_j\equiv1\pmod{p}$ for all $1\leq j\leq r$. In other words, the parts of $\gamma$ are the multiplicative inverses modulo $p$ of the respective parts of $\lambda$. The equivalence class $[\gamma]\in\mathcal{P}_r/_{\sim_p'}$ is the inverse element of $[\lambda]$, since
\begin{align*}
[\lambda]\cdot[\gamma]=\left[\left(\lambda_1\gamma_1,\lambda_2\gamma_2,\dots,\lambda_r\gamma_r\right)\right]=[(\underbrace{1,\dots,1}_{r\text{ times}})].
\end{align*}
It is straightforward to show that component-wise multiplication is both associative and commutative, since this operation does not depend on the choice of equivalence class representative and simply consists of integer multiplication in each component of the representative partitions. Thus $\mathcal{P}_r/_{\sim_p'}$ is an abelian group under component-wise multiplication.
\end{proof}

We now introduce the type of reduction required to obtain an abelian group structure on partitions, as opposed to equivalence classes of partitions, under component-wise multiplication.

\begin{definition}
Let $k$ be any positive integer. We define the unary operation \textit{component-wise zero-reduction modulo $k$} on any partition $\lambda\in\mathcal{P}$ as follows.
\begin{itemize}
\item If $\lambda$ has some part divisible by $k$, then $\lambda$ reduces to $\emptyset$.
\item If $\lambda$ has no part divisible by $k$, then we perform component-wise reduction modulo $k$ on $\lambda$ as defined in Definition \ref{componentwise_reduction_defn}.
\end{itemize}
We call a partition $\lambda\in\mathcal{P}$ \textit{component-wise zero-reduced modulo $k$} if $\lambda$ is $\emptyset$, or if $\lambda$ has no parts divisible by $k$ and is component-wise reduced modulo $k$.
\end{definition}

Component-wise zero-reduction modulo $k$ reduces a partition $\lambda\in\mathcal{P}$ with no parts divisible by $k$ to the partition $\lambda'$ of minimal size such that each part of $\lambda'$ is congruent modulo $k$ to the corresponding part of $\lambda$.  However, if a partition $\lambda$ has a part divisible by $k$, then component-wise zero-reduction modulo $k$ reduces $\lambda$ to $\emptyset$.  In other words, a component-wise zero-reduced modulo $k$ partition either is $\emptyset$ or has no parts divisible by $k$, difference between any two consecutive parts less than $k$, and smallest part less than $k$.

\begin{definition}\label{componentwise_mult_defn}
For any positive integer $k$, we define the binary operation \textit{component-wise multiplication and zero-reduction modulo $k$}, denoted $\cdot_k^{(0)}$, on $\mathcal{P}$ as follows.  Let $\lambda=\left(\lambda_1,\lambda_2,\dots,\lambda_\ell\right)$ and $\pi=\left(\pi_1,\pi_2,\dots,\pi_i\right)$ be any two partitions in $\mathcal{P}$ with $\ell\geq i$.  Define $\lambda\cdot_k^{(0)}\pi$ by first component-wise multiplying the parts of $\lambda$ and $\pi$: $$(\lambda_1\pi_1,\lambda_2\pi_2,\dots,\lambda_i\pi_i,\underbrace{0,\dots,0}_{\ell-i})=\left(\lambda_1\pi_1,\lambda_2\pi_2,\dots,\lambda_i\pi_i\right),$$ and then component-wise zero-reducing the result modulo $k$.
\end{definition}

Recall that we use the definition $\pi_j=0$ for all $j>\ell(\pi)$ when component-wise multiplying two partitions of different length.

Recall also that the set $\mathcal{P}_{k,r}^{(0)}$ is the set of length-$r$ partitions with no part divisible by $k$, where the difference between any two consecutive parts is less than $k$ and the smallest part is less than $k$. We now prove Theorem \ref{mult_group_structure}: the set $\mathcal{P}_{p,r}^{(0)}$ is an abelian group under $\cdot_p^{(0)}$, for any prime $p$.

\begin{proof}[Proof of Theorem \ref{mult_group_structure}]
We prove the result by showing that the set $\mathcal{P}_{p,r}^{(0)}$ with the operation of component-wise multiplication and zero-reduction modulo $p$ is isomorphic to the abelian group $\mathcal{P}_r/_{\sim_p'}$.

Let $p$ be a prime, and fix a non-negative integer $r$. Let $\lambda\in\mathcal{P}$ be any length-$r$ partition. Define the map $\psi$ on $\mathcal{P}_r/_{\sim_p'}$ by $\psi([\lambda]):=\gamma$, where $\gamma$ is the result of component-wise zero-reduction modulo $p$ on $\lambda$.  If $\lambda$ has a part divisible by $p$, then $\gamma=\emptyset$. Otherwise, we have $\gamma=\left(\gamma_1,\gamma_2,\dots,\gamma_r\right)$, where $\gamma_r$ and $\gamma_j$, $1\leq j<r$, are defined the same as in \eqref{gamma_ell_defn} and \eqref{gamma_j_defn}, respectively, with $k=p$. Let $\pi\in\mathcal{P}$ be any partition, and suppose that $\psi([\pi]):=\nu$. If $\nu=\gamma$, then either we have that $\nu=\gamma=\emptyset$, in which case $\lambda$ and $\pi$ both have a part divisible by $p$ and therefore $[\lambda]=[\pi]=[\emptyset]$, or we have that $\nu=\gamma\neq\emptyset$, in which case the parts of $\gamma$ and $\nu$ satisfy \eqref{ell} and \eqref{other}, with $k=p$. In the latter case, we see that if $\nu=\gamma$, then $\pi_j\equiv\lambda_j\pmod{p}$ for all $1\leq j\leq r$, and therefore $\pi$ and $\lambda$ are component-wise congruent modulo $p$.  Thus, we have shown that $\psi([\lambda])=\psi([\pi])$ implies $[\pi]=[\lambda]$, and so $\psi$ is one-to-one.

Now, observe that the image of each $[\lambda]\in\mathcal{P}_r/_{\sim_p'}$ is a component-wise zero-reduced modulo $p$ partition, so $\psi$ maps to $\mathcal{P}_{p,r}^{(0)}$.  Let $\mu\in\mathcal{P}_{p,r}^{(0)}$.  Then we have that $\psi([\mu])=\mu$, since $\mu$ is already component-wise zero-reduced modulo $p$.  Thus, we have shown that $\psi:\mathcal{P}_r/_{\sim_p'}\to\mathcal{P}_{p,r}^{(0)}$ is onto.

Finally, we show that $\psi([\lambda]\cdot[\pi])=\psi([\lambda])\cdot_p^{(0)}\phi([\pi])$ for any $[\lambda],[\pi]\in\mathcal{P}_r/_{\sim_p'}$. If any part of $\lambda$ or $\pi$ is divisible by $p$, then the left and right sides are both equal to the empty partition $\emptyset$. Suppose $\lambda$ and $\pi$ have no parts divisible by $p$. On the left, we have that
\begin{align*}
\psi([\lambda]\cdot[\pi])&=\psi\left(\left[\left(\lambda_1\pi_1,\lambda_2\pi_2,\dots,\lambda_r\pi_r\right)\right]\right)\\
&=\rho:=\left(\rho_1,\rho_2,\dots,\rho_r\right),
\end{align*}
where $$\rho_r:=\lambda_r\pi_r-\left\lfloor\frac{\lambda_r\pi_r}{p}\right\rfloor p$$ and $$\rho_j:=\lambda_j\pi_j-\left\lfloor\frac{\lambda_j\pi_j-\rho_{j+1}}{p}\right\rfloor p$$ for each $1\leq j<r$. On the other hand, we have that $\psi([\lambda])\cdot_p^{(0)}\psi([\pi])=\gamma\cdot_p^{(0)}\nu$, where the parts of $\gamma$ and $\nu$ are defined in \eqref{gamma_ell_defn}, \eqref{gamma_j_defn}, \eqref{ell}, and \eqref{other}, with $k=p$.  Then we have that $\psi([\lambda])\cdot_p^{(0)}\psi([\pi]):=\zeta=\left(\zeta_1,\zeta_2,\dots,\zeta_r\right)$, where $$\zeta_r:=\gamma_r\nu_r-\left\lfloor\frac{\gamma_r\nu_r}{p}\right\rfloor p$$ and $$\zeta_j:=\gamma_j\nu_j-\left\lfloor\frac{\gamma_j\nu_j-\zeta_{j+1}}{p}\right\rfloor p$$ for each $1\leq j<r$. We see that $\rho_j=\zeta_j$ for all $1\leq j\leq r$, because both $\rho_j$ and $\zeta_j$ are component-wise zero-reduced modulo $p$ and $\lambda_j\pi_j\equiv\gamma_j\nu_j\pmod{p}$. Thus, we have shown that $\rho=\zeta$, and therefore $\psi$ is an isomorphism. This completes the proof that $\mathcal{P}_{p,r}^{(0)}$ is an abelian group.
\end{proof}

\begin{remark}
We note here that $\mathcal{P}_{p,r}^{(0)}$ is a finite abelian group. For any given prime $p$ and any non-negative integer $r<p$, the partition in $\mathcal{P}_{p,r}^{(0)}$ of largest size is the partition $$\lambda_{\max}:=\left(r(p-1),(r-1)(p-1),\dots,2(p-1),p-1\right),$$
which has size $(r+(r-1)+\cdots+2+1)(p-1)=r(r+1)(p-1)/2$. If $r>p$, then this partition reduces to $\emptyset$ since the part $p(p-1)$ is divisible by $p$, so the partition in $\mathcal{P}_{p,r}^{(0)}$ of largest size has parts strictly smaller than the respective parts of $\lambda_{\max}$. In either case, any partition of larger size must have smallest part at least $p$, or a difference of at least $p$ between some pair of consecutive parts, or some part divisible by $p$, or more than $r$ parts; such partitions are not contained in $\mathcal{P}_{p,r}^{(0)}$. In general, each non-empty partition $\lambda=\left(\lambda_1,\lambda_2,\dots,\lambda_r\right)\in\mathcal{P}_{p,r}^{(0)}$, $r<p$, has the $p-1$ choices $1,2,\dots,p-1$ for the smallest part $\lambda_r$, and then the $p$ choices $\lambda_j+0,\lambda_j+1,\lambda_j+2,\dots,\lambda_j+p-1$ for the part $\lambda_{j-1}$, for all $1<j\leq r$. Counting the total number of possibilities, we see that if $r<p$, then $$\left|\mathcal{P}_{p,r}^{(0)}\right|=(p-1)p^{r-1}+1.$$ If $r\geq p$, then the number of partitions in $\mathcal{P}_{p,r}^{(0)}$ is less than $(p-1)p^{r-1}+1$.
\end{remark}

%%%%%%%%%%%%%%%%%%%%%%%%%%%%%%%%%%%%%% 

\section{Vector Space Structures on Reduced Partitions}\label{section_vector_spaces}

In this section, we prove two vector space structures on partitions.

\subsection{Vector Space Structure on $\mathcal{P}_{\cup_p}$}

First, we prove that the abelian group $\mathcal{P}_{\cup_p}$, for any prime $p$, is also a vector space over the finite field $\mathbb{Z}_p$.

\begin{proof}[Proof of Theorem \ref{vector_space_structures}, Part 1]
    Let $a,b\in \mathbb{Z}_p$ and $\lambda = \left(\lambda_1, \lambda_2, \dots, \lambda_\ell\right),\pi=\left(\pi_1,\pi_2,\dots,\pi_i\right)\in\mathcal{P}_{\cup_p}$. We define scalar multiplication by $a\lambda:=\underbrace{\lambda\cup_p\cdots\cup_p\lambda}_{a\text{ times}}\in\mathcal{P}_{\cup_p}$.

We have that $1\lambda=\lambda$, since $1\lambda$ is defined to be the result of multiplicity-reduction modulo $k$ on $\lambda$, which is already multiplicity-reduced modulo $k$.
        
Let $\alpha$ be a part of $\lambda$. Recall that $m_\lambda(\alpha)\geq0$ denotes the multiplicity of the part $\alpha$ in $\lambda$. Note that $m_\lambda(\alpha) < p$, since $\lambda$ is multiplicity-reduced modulo $p$. We have that $m_{(ab)\lambda}(\alpha)=M_1$, where $M_1$ is the unique non-negative integer such that $M_1\equiv (ab)m_\lambda(\alpha)\pmod{p}$ and $M_1<p$. On the other hand, $m_{b\lambda}(\alpha)=M_2$, where $M_2$ is the unique non-negative integer such that $M_2\equiv bm_\lambda(\alpha)\pmod{p}$ and $M_2<p$, and so $m_{a(b\lambda)}(\alpha)=M_3$, where $M_3$ is the unique non-negative integer such that $M_3\equiv a\left(bm_\lambda(\alpha)\right)\pmod{p}$ and $M_3<p$. Since multiplication in $\mathbb{Z}_p$ is associative, we have that $M_1=M_3$, so $m_{(ab)\lambda}(\alpha)=m_{a(b\lambda)}(\alpha)$.  Therefore, since $\alpha$ is an arbitrary part of $\lambda$, we have shown that $(ab)\lambda=a(b\lambda)$.
        
    Now, for any part $\beta$ of $\lambda$ or $\pi$, we have that $m_{\lambda\cup_p\pi}(\beta)=M_4$, where $M_4$ is the unique non-negative integer such that $M_4\equiv m_\lambda(\beta)+m_\pi(\beta)\pmod{p}$ and $M_4<p$.  Then we calculate that $m_{a\left(\lambda\cup_p\pi\right)}(\beta)=M_5$, where $M_5$ is the unique non-negative integer such that $M_5\equiv a\left(m_\lambda(\beta)+m_\pi(\beta)\right)\pmod{p}$ and $M_5<p$.  Similarly, we define $M_6$ as the unique non-negative integer such that $M_6\equiv am_\lambda(\beta)\pmod{p}$ and $M_6<p$, and we define $M_7$ as the unique non-negative integer such that $M_7\equiv am_\pi(\beta)\pmod{p}$ and $M_7<p$.  Then we see that $m_{(a\lambda)\cup_p(a\pi)}(\beta)=M_8$, where $M_8$ is the unique non-negative integer such that $M_8\equiv am_\lambda(\beta)+am_\pi(\beta)\pmod{p}$ and $M_8<p$. Since the distributive laws hold in $\mathbb{Z}_p$, we have that $M_5=M_8$ and therefore $m_{a\left(\lambda\cup_p\pi\right)}(\beta)=m_{(a\lambda)\cup_p(a\pi)}(\beta)$.  We have shown that $a\left(\lambda\cup_p\pi\right)=(a\lambda)\cup_p(a\pi)$. 
        
Finally, we recall that $\alpha$ is an arbitrary element of $\lambda$, and we compare $m_{(a+b)\lambda}(\alpha)$ and $m_{(a\lambda)\cup_p(b\lambda)}(\alpha)$. We observe that $m_{(a+b)\lambda}(\alpha)=M_9$, where $M_9$ is the unique non-negative integer such that $M_9\equiv(a+b)m_\lambda(\alpha)\pmod{p}$ and $M_9<p$.  Now, we have that $m_{a\lambda}(\alpha)=M_{10}$, where $M_{10}$ is the unique non-negative integer such that $M_{10}\equiv am_\lambda(\alpha)\pmod{p}$ and $M_{10}<p$. Then $m_{(a\lambda)\cup_p(b\lambda)}(\alpha)=M_{11}$, where $M_{11}$ is the smallest non-negative integer such that $M_{11}\equiv M_{10}+M_2\equiv am_\lambda(\alpha)+bm_\lambda(\alpha)\pmod{p}$.  Again, since the distributive laws hold in $\mathbb{Z}_p$, we have that $M_9=M_{11}$ and therefore $m_{(a+b)\lambda}(\alpha)=m_{(a\lambda)\cup_p(b\lambda)}(\alpha)$.  Thus, we have that $(a+b)\lambda=(a\lambda)\cup_p(b\lambda)$.  We have now shown that $\mathcal{P}_{\cup_p}$ is a vector space over $\mathbb{Z}_p$.
\end{proof}

We identify a basis for the vector space $\mathcal{P}_{\cup_p}$ over $\mathbb{Z}_p$ in the next theorem.

\begin{theorem}\label{basis_thm}
The set $\{\gamma\in\mathcal{P}:\ell(\gamma)=1\}$ consisting of all partitions with only a single part forms a basis for the vector space $\mathcal{P}_{\cup_p}$ over $\mathbb{Z}_p$. 
\end{theorem}
\begin{proof}
    Let $\lambda = \left(\lambda_1,\lambda_2,\dots, \lambda_\ell\right)\in\mathcal{P}_{\cup_p}$. Then each part of $\lambda$ has multiplicity less than $p$. Observe that $\lambda_i\geq0$ for all $1\leq i\leq\ell$, with equality if and only if $\lambda=\emptyset$. Then the length-1 partitions $\left(\lambda_i\right)$, $1\leq i\leq\ell$, formed by the individual parts of $\lambda$ are all elements of the set $\{\gamma\in\mathcal{P}:\ell(\gamma)=1\}$. Consider the concatenation and reduction modulo $p$ of all such length-1 partitions originating from the parts of $\lambda$: $\left(\lambda_1\right) \cup_p \left(\lambda_2\right)\cup_p\cdots \cup_p \left(\lambda_\ell\right)$. Since the parts of $\lambda$ have multiplicity less than $p$, the reduction step does not remove any parts of the concatenation and reduction modulo $p$. Thus, we see that $\left(\lambda_1\right) \cup_p \left(\lambda_2\right)\cup_p\cdots \cup_p \left(\lambda_\ell\right) = \left(\lambda_1,\lambda_2, \dots, \lambda_\ell\right)=\lambda$. This proves that the set $\{\gamma\in\mathcal{P}:\ell(\gamma)=1\}$ spans $\mathcal{P}_{\cup_p}$.
    
    Now, for some positive integer $k$, consider any $k$ distinct elements $\left(\pi_1\right),\left(\pi_2\right),\dots,\left(\pi_k\right)$ in the set $\{\gamma\in\mathcal{P}:\ell(\gamma)=1\}$, and assume that $a_1\left(\pi_1\right)\cup_p a_2\left(\pi_2\right)\cup_p \cdots \cup_p a_k\left(\pi_k\right)=\emptyset$, where $a_1,a_2, \dots, a_k\in \mathbb{Z}_p$. For the assumed equation to hold, the partition on the left side must reduce to $\emptyset$ after concatenation and reduction modulo $p$ or be equal to $\emptyset$ from the start. Since the length-1 partitions $\left(\pi_i\right)$, $1\leq i \leq k$, are distinct and non-empty, we must have that $a_i\equiv 0 \pmod{p}$ for all $1\leq i\leq k$. Since $a_i\in\mathbb{Z}_p$, this implies that $a_i = 0$ for each $1\leq i\leq k$. Thus, the set $\{\gamma\in\mathcal{P}:\ell(\gamma)=1\}$ is linearly independent and is therefore a basis for the vector space $\mathcal{P}_{\cup_p}$ over $\mathbb{Z}_p$. 
\end{proof}

The following corollary of Theorem \ref{basis_thm} gives infinitely many examples of partition subspaces.

\begin{corollary}\label{subspace_cor}
Let $p$ be a prime.  Let $m,a_1,a_2,\dots,a_m$ be any positive integers, and define $A_m:=\left\{a_1,a_2,\dots,a_m\right\}$. The subset of partitions generated by the set $\left\{\left(\alpha\right)\in\mathcal{P}:\alpha\in A_m\right\}$ forms a subspace of the vector space $ \mathcal{P}_{\cup_p}$ over $\mathbb{Z}_p$. 
\end{corollary}

\begin{proof}
The corollary follows immediately from Theorem \ref{basis_thm}, because any set of the form $\left\{\left(\alpha\right)\in\mathcal{P}:\alpha\in A_m\right\}$ is a subset of the basis for the vector space $\mathcal{P}_{\cup_p}$ over $\mathbb{Z}_p$. 
\end{proof}

As a direct consequence of Corollary \ref{subspace_cor}, we obtain another abelian group of partitions under the operation of concatenation and reduction modulo $p$.

\begin{corollary}\label{subgroup_cor}
Let $p$ be a prime.  Let $m,a_1,a_2,\dots,a_m$ be any positive integers, and define $A_m:=\left\{a_1,a_2,\dots,a_m\right\}$. The subset of partitions generated by $\left\{\left(\alpha\right)\in\mathcal{P}:\alpha\in A_m\right\}$ forms a subgroup of $\mathcal{P}_{\cup_p}$.
\end{corollary}

The proof of Corollary \ref{subgroup_cor} is immediate, so we omit it.

\subsection{Vector Space Structure on $\mathcal{P}_k^+$}

Here, we prove that for any positive integer $k$, the abelian group $\mathcal{P}_k^+$ is also a vector space over the finite field $\mathbb{Z}_p$, for any prime $p$.

\begin{proof}[Proof of Theorem \ref{vector_space_structures}, Part 2]
We prove first that $\mathcal{P}/_{\sim_k}$ is a vector space, and then that $\mathcal{P}_k^+$ is isomorphic to $\mathcal{P}/_{\sim_k}$ via a vector space isomorphism.

Let $a,b \in \mathbb{Z}_p$, and let $[\lambda] = \left[\left(\lambda_1,\dots,\lambda_\ell\right)\right]$ and $[\pi] = \left[\left(\pi_1, \dots , \pi_i\right)\right]$ be equivalence classes of partitions under component-wise congruence modulo $k$. Assume without loss of generality that $\ell \geq i$. We define scalar multiplication by $a[\lambda] = \underbrace{[\lambda] + \cdots + [\lambda]}_{a\text{ times}}$, which yields another equivalence class of partitions.  It is clear that $1[\lambda]=[\lambda]$ by definition.

We first show that $a(b[\lambda])=(ab)[\lambda]$.  We calculate that
\begin{align*}
a(b[\lambda])&=a(\underbrace{[\lambda]+\dots+[\lambda]}_{b\text{ times}})\\
&=a\left(\left[\left(b\lambda_1,b\lambda_2,\dots,b\lambda_\ell\right)\right]\right)\\
&=\underbrace{\left[\left(b\lambda_1,b\lambda_2,\dots,b\lambda_\ell\right)\right]+\cdots+\left[\left(b\lambda_1,b\lambda_2,\dots,b\lambda_\ell\right)\right]}_{a\text{ times}}\\
&=\left[\left(a\left(b\lambda_1\right),a\left(b\lambda_2\right),\dots,a\left(b\lambda_\ell\right)\right)\right]\\
&=\left[\left((ab)\lambda_1,(ab)\lambda_2,\dots,(ab)\lambda_\ell\right)\right]\\
&=\underbrace{[\lambda]+\cdots+[\lambda]}_{ab\text{ times}}=(ab)[\lambda].
\end{align*}

A similar straightforward computation shows that $(a+b)[\lambda]=(a[\lambda])+(b[\lambda])$. Finally, we recall that $\pi_j=0$ for any $j>\ell(\pi)$, and we show that $a\left([\lambda]+[\pi]\right)=(a[\lambda])+(a[\pi])$:
   \begin{align*}
       a([\lambda]+[\pi]) &=a\left(\left[\lambda_1+\pi_1,\lambda_2+\pi_2,\dots,\lambda_\ell+\pi_\ell\right]\right)\\
       &=\underbrace{\left[\lambda_1+\pi_1,\lambda_2+\pi_2,\dots,\lambda_\ell+\pi_\ell\right]+\cdots+\left[\lambda_1+\pi_1,\lambda_2+\pi_2,\dots,\lambda_\ell+\pi_\ell\right]}_{a\text{ times}}\\
       &=\left[\left(a\left(\lambda_1+\pi_1\right),a\left(\lambda_2+\pi_2\right),\dots,a\left(\lambda_\ell+\pi_\ell\right)\right)\right]\\
    &=\left[\left(a\lambda_1+a\pi_1,a\lambda_2+a\pi_2,\dots,a\lambda_\ell+a\pi_\ell\right)\right]\\
       &=\left[\left(a\lambda_1,a\lambda_2,\dots,a\lambda_\ell\right)\right]+\left[\left(a\pi_1,a\pi_2,\dots,a\pi_i\right)\right]\\
       &=\underbrace{[\lambda]+\cdots+[\lambda]}_{a\text{ times}}+\underbrace{[\pi]+\cdots+[\pi]}_{a\text{ times}}=a[\lambda]+a[\pi].
   \end{align*}
   
   Now, define the same map $\varphi:\mathcal{P}/_{\sim_k}\to\mathcal{P}_k^+$ as in the proof of Theorem \ref{group_structures}, Part 2, by $\varphi([\lambda])=\gamma:=\left(\gamma_1,\gamma_2,\dots,\gamma_\ell\right)$, where $\gamma_\ell$ and $\gamma_j$ for $1\leq j<\ell$ are defined in \eqref{gamma_ell_defn} and \eqref{gamma_j_defn}, respectively.  We have already shown that $\varphi$ is a group isomorphism, which implies that $\mathcal{P}_k^+$ is an abelian group and that $\varphi([\lambda]+[\pi])=\varphi([\lambda])+_k\varphi([\pi])$.  Now, define scalar multiplication in $\mathcal{P}_k^+$ by $a\lambda:=\underbrace{\lambda+_k\cdots+_k\lambda}_{a\text{ times}}$. It remains to prove that $\varphi(a[\lambda])=a\varphi([\lambda])$. We calculate that
   \begin{align*}
   a\varphi([\lambda])&=a\gamma\\
   &=\underbrace{\gamma+_k\cdots+_k\gamma}_{a\text{ times}}\\
   &\sim_k\left(a\left(\lambda_1-\left\lfloor\frac{\lambda_1-\gamma_2}{k}\right\rfloor k\right),\dots,a\left(\lambda_{\ell-1}-\left\lfloor\frac{\lambda_{\ell-1}-\gamma_\ell}{k}\right\rfloor k\right),a\left(\lambda_\ell-\left\lfloor\frac{\lambda_\ell}{k}\right\rfloor k\right)\right),
   \end{align*}
   which is then component-wise reduced modulo $k$, by the definition of the operation $+_k$. On the other hand, we have that
   \begin{align*}
   \varphi(a[\lambda])&=\varphi(\underbrace{[\lambda]+\cdots+[\lambda]}_{a\text{ times}})\\
   &=\varphi\left(\left[\left(a\lambda_1,a\lambda_2,\dots,a\lambda_\ell\right)\right]\right)\\
   &=\left(\xi_1,\xi_2,\dots,\xi_\ell\right)\\
   &:=\left(a\lambda_1-\left\lfloor\frac{a\lambda_1-a\xi_2}{k}\right\rfloor k,\dots,a\lambda_{\ell-1}-\left\lfloor\frac{a\lambda_{\ell-1}-a\xi_\ell}{k}\right\rfloor k,a\lambda_\ell-\left\lfloor\frac{a\lambda_\ell}{k}\right\rfloor k\right),
   \end{align*}
   where we define $\xi:=\left(\xi_1,\xi_2,\dots,\xi_\ell\right)$ to be the image of $\left[\left(a\lambda_1,a\lambda_2,\dots,a\lambda_\ell\right)\right]$ under $\varphi$. The image $\varphi(a[\lambda])$ is also component-wise reduced modulo $k$, by the definition of the map $\varphi$.  Therefore, since both $a\varphi([\lambda])$ and $\varphi(a[\lambda])$ are partitions which are component-wise reduced modulo $k$ and whose respective parts are congruent modulo $k$, we have $a\varphi([\lambda])=\varphi(a[\lambda])$. This completes the proof that $\mathcal{P}_k^+$ is a vector space.
\end{proof}

%%%%%%%%%%%%%%%%%%%%%%%%%%%%%%%%%%%%%%

\section{Ring Structure on $\mathcal{P}_k^+$}\label{section_ring}

We prove here that $\mathcal{P}_k^+$ forms a commutative ring under component-wise addition and component-wise multiplication modulo $k$, where component-wise multiplication modulo $k$ is defined below.

\begin{definition}\label{componentwise_mult_zeros-okay_defn}
For any positive integer $k$, we define the binary operation \textit{component-wise multiplication and reduction modulo $k$}, denoted $\cdot_k$, on $\mathcal{P}$ as follows. Let $\lambda=\left(\lambda_1,\lambda_2,\dots,\lambda_\ell\right)$ and $\pi=\left(\pi_1,\pi_2,\dots,\pi_i\right)$ be any two partitions in $\mathcal{P}$ with $\ell\geq i$. Define $\lambda\cdot_k\pi$ by first component-wise multiplying the parts of $\lambda$ and $\pi$: $$(\lambda_1\pi_1,\lambda_2\pi_2,\dots,\lambda_i\pi_i,\underbrace{0,\dots,0}_{\ell-i})=\left(\lambda_1\pi_1,\lambda_2\pi_2,\dots,\lambda_i\pi_i\right),$$ and then component-wise reducing the result modulo $k$ as described in Definition \ref{componentwise_reduction_defn}.
\end{definition}

Unlike component-wise multiplication and zero-reduction modulo $k$, this new operation of component-wise multiplication and reduction modulo $k$ allows nonzero partitions to have parts divisible by $k$. For this reason, the set $\mathcal{P}_k^+$ admits zero divisors with respect to $\cdot_k$, but, as we will show in the following proof, we still have that $\left(\mathcal{P}_k^+,+_k,\cdot_k\right)$ satisfies all required axioms to be a commutative ring.

\begin{proof}[Proof of Theorem \ref{ring_structure}]
By Theorem \ref{group_structures}, Part 1, we have that $\mathcal{P}_k^+$ is an abelian group under $+_k$.  It is straightforward to see from Definition \ref{componentwise_mult_zeros-okay_defn} that $\cdot_k$ is commutative. To prove that $\mathcal{P}_k^+$ is a ring, it remains to show that $\cdot_k$ is associative and that the distributive law holds in $\mathcal{P}_k^+$ with the operations $+_k$ and $\cdot_k$.

Let $\lambda = \left(\lambda_1, \dots, \lambda_\ell\right)$, $\pi=\left(\pi_1, \dots, \pi_i\right)$, and $\tau=\left(\tau_1, \dots, \tau_t\right)$ be three partitions in $\mathcal{P}$, and assume without loss of generality that $\ell\leq i\leq t$. We have that
\begin{align*}
(\lambda\cdot_k\pi)\cdot_k\tau&\sim_k\left(\lambda_1\pi_1,\lambda_2\pi_2,\dots,\lambda_i\pi_i\right)\cdot_k\left(\tau_1,\tau_2,\dots,\tau_t\right)\\
&\sim_k\left(\lambda_1\pi_1\tau_1,\lambda_2\pi_2\tau_2,\dots,\lambda_t\pi_t\tau_t\right)\\
&\sim_k\left(\lambda_1,\lambda_2,\dots,\lambda_\ell\right)\cdot_k\left(\pi_1\tau_1,\pi_2\tau_2,\dots,\pi_t\tau_t\right)\\
&\sim_k\lambda\cdot_k(\pi\cdot_k\tau).
\end{align*}
Since the operation $\cdot_k$ yields a partition which is component-wise reduced modulo $k$, this implies that $(\lambda\cdot_k\pi)\cdot_k\tau=\lambda\cdot_k(\pi\cdot_k\tau)$, and therefore $\cdot_k$ is associative.

Recall that we consider the $j$-th part of a partition $\gamma$ to be equal to zero if $j>\ell(\gamma)$. We also have that
\begin{align*}
\lambda\cdot_k(\pi+_k\tau)&\sim_k\lambda\cdot_k\left(\pi_1+\tau_1,\pi_2+\tau_2,\dots,\pi_i+\tau_i\right)\\
&\sim_k\left(\lambda_1\left(\pi_1+\tau_1\right),\lambda_2\left(\pi_2+\tau_2\right),\dots,\lambda_i\left(\pi_i+\tau_i\right)\right)\\
&\sim_k\left(\lambda_1\pi_1+\lambda_1\tau_1,\lambda_2\pi_2+\lambda_2\tau_2,\dots,\lambda_i\pi_i+\lambda_i\tau_i\right)\\
&\sim_k(\lambda\cdot_k\pi)+_k(\lambda\cdot_k\tau).
\end{align*}
Again, we see that $\lambda\cdot_k(\pi+_k\tau)=(\lambda\cdot_k\pi)+_k(\lambda\cdot_k\tau)$, since both sides are partitions which are component-wise reduced modulo $k$ and which are component-wise congruent modulo $k$ to each other. Thus, the distributive law holds in $\mathcal{P}_k^+$. This completes the proof that $\mathcal{P}_k^+$ is a commutative ring.
\end{proof}

We also determine the structure of all finitely generated ideals of $\mathcal{P}_p^+$, for any prime $p$.

\begin{theorem}\label{ideal_thm}
For any prime $p$, every finitely generated ideal of $\mathcal{P}_p^+$ is principal.
\end{theorem}

\begin{proof}
Let $I$ be a finitely generated ideal of $\mathcal{P}_p^+$. Then we have that $\langle\pi\rangle\subseteq I$ for any $\pi\in I$. If $I\subseteq\langle\lambda\rangle$ for some $\lambda\in I$, then $I$ is principal. Let $\lambda=\left(\lambda_1,\lambda_2,\dots,\lambda_\ell\right)\in I$, and suppose by way of contradiction that there exists $\tau=\left(\tau_1,\tau_2,\dots,\tau_t\right)\in I$ such that $\tau_i\not\equiv0\pmod{\lambda_i}$ for some $i\geq1$. Then $I$ contains all partitions of the form $$(r\cdot_p\lambda)+_p(s\cdot_p\tau)=(r_1\lambda_1+s_1\tau_1,r_2\lambda_2+s_2\tau_2,\dots,r_y\lambda_y+s_y\tau_y),$$ where $r=(r_1,r_2,\dots,r_c),s=(s_1,s_2,\dots,s_h)\in\mathcal{P}_p^+$ and $y=\max\{\ell,t,c,h\}$. Define the partition $g:=(g_1,g_2,\dots,g_y)\in\mathcal{P}_p^+$ with the property that $g_i \equiv \gcd(\lambda_i,\tau_i)\pmod{p}$ for all $1 \leq i \leq y.$ Then for each $i\geq1$, there exist positive integers $a_i,b_i$ such that
\begin{align}\label{g_congruence_classes}
\lambda_i\equiv a_ig_i\pmod{p}\quad\text{and}\quad\tau_i\equiv b_ig_i\pmod{p}.
\end{align}
Therefore, we have that
\begin{align*}
(r\cdot_p\lambda)+_p(s\cdot_p\tau)&\sim_p\left(r_1(a_1g_1)+s_1(b_1g_1),r_2(a_2g_2)+s_2(b_2g_2),\dots,r_y(a_yg_y)+s_y(b_yg_y)\right)\\
&\sim_p\left(g_1(r_1a_1+s_1b_1),g_2(r_2a_2+s_2b_2),\dots,g_y(r_ya_y+s_yb_y)\right)\\
&\sim_pg\cdot_p\nu,
\end{align*}
where $\nu:=\left(\nu_1,\nu_2,\dots,\nu_y\right)\in\mathcal{P}_p^+$ is defined so that $\nu_i\equiv r_ia_i+s_ib_i\pmod{p}$ for all $1\leq i\leq y$. Since the operations $+_p$ and $\cdot_p$ both yield component-wise reduced modulo $p$ partitions, we have that $(r\cdot_p\lambda)+_p(s\cdot_p\tau)=g\cdot_p\nu$. It follows that $\langle\lambda,\tau\rangle\subseteq\langle g\cdot_p\nu\rangle\subseteq\langle g\rangle$, since $g\cdot_p\nu\in\langle g\rangle$. Let $\mu\cdot_p g\sim_p\left(\mu_1g_1,\mu_2g_2,\dots,\mu_mg_m\right)$ be an arbitrary partition in $\langle g\rangle$. Then we can use \eqref{g_congruence_classes} to rewrite $$\mu\cdot_p g=\left(\mu_1\left(a_1^{-1}\lambda_1\right),\mu_2\left(a_2^{-1}\lambda_2\right),\dots,\mu_m\left(a_m^{-1}\lambda_m\right)\right),$$ where $a_i^{-1},b_i^{-1}$ are the multiplicative inverses of $a_i,b_i$ in $\mathbb{Z}_p$, for each $1\leq i\leq m$. Now, define $\gamma:=\left(\gamma_1,\gamma_2,\dots,\gamma_m\right)\in\mathcal{P}_p^+$ by $\gamma_i\equiv\mu_ia_i^{-1}\pmod{p}$ for each $1\leq i\leq m$. Then we have that $\mu\cdot_p g=\gamma\cdot_p\lambda$, and therefore $\mu\cdot_p g\in\langle\lambda\rangle$. We have shown that $\langle \lambda,\tau\rangle\subseteq\langle g\rangle\subseteq\langle\lambda\rangle$, which contradicts the assumption that there exists $\tau\in I$ such that $\tau\not\in\langle\lambda\rangle$.  Thus, $I=\langle\lambda\rangle$ is a principal ideal. This process can be repeated for any finite number of generating partitions of $I$, so every finitely generated ideal of $\mathcal{P}_p^+$ is principal.
\end{proof}

Theorem \ref{ideal_thm} shows that for a prime $p$, any finitely generated ideal of partitions with smallest part less than $p$ and difference between any two consecutive parts less than $p$ consists entirely of component-wise products of a single generating partition. In every such ideal, the unique generating partition is the partition whose parts are congruent modulo $p$ to the respective components of the $\ell$-tuple of greatest common divisors of the generating partitions, where $\ell$ is the maximum length of the generating partitions.

\begin{example}
Let $p=3$, $\lambda=(3,2)$, $\pi=(5,4,3,1)$, and $\tau=(9,8,6,4,2)$, and consider the ideal $I=\langle\lambda,\pi,\tau\rangle\subset\mathcal{P}_3^+$. The 5-tuple of component-wise greatest common divisors of $\lambda,\pi,\tau$ is $g=(\gcd(3,5,9),\gcd(2,4,8),\gcd(0,3,6),\gcd(0,1,4),\gcd(0,0,2))=(1,2,3,1,2)$, so the unique generating partition of $I$ is $\gamma=(10,8,6,4,2)$. To illustrate that $I=\langle\gamma\rangle$, we calculate that
\begin{align*}
(3,1)\cdot_p(10,8,6,4,2)&=(3,2);\\
(2,2,1,1)\cdot_p(10,8,6,4,2)&=(5,4,3,1);\\
(3,1,1,1,1)\cdot_p(10,8,6,4,2)&=(9,8,6,4,2).
\end{align*}
\end{example}

%%%%%%%%%%%%%%%%%%%%%%%%%%%%%%%%%%%%%%

\backmatter

\section*{Statements and Declarations}

\begin{itemize}
\item Funding: This work was supported by NSF Grant DMS-2149921.
\item Competing interests: The authors have no competing interests to declare that are relevant to the content of this article.
\item Data Availability: No datasets were generated or analyzed during this study.
\end{itemize}

%%%%%%%%%%%%%%%%%%%%%%%%%%%%%%%%%%%%%%

\end{document}